\documentclass[12pt,twoside]{amsart}
\usepackage{mabliautoref}
\usepackage{amssymb,amsthm,amsmath}
\RequirePackage[dvipsnames,usenames]{xcolor}
\usepackage{hyperref}
\usepackage{mathtools}
\usepackage[abbrev,alphabetic]{amsrefs}
\usepackage[all]{xy}
\usepackage{tikz}
\usepackage{tikz-cd}

\hypersetup{
	bookmarks,
	bookmarksdepth=3,
	bookmarksopen,
	bookmarksnumbered,
	pdfstartview=FitH,
	colorlinks,backref,hyperindex,
	linkcolor=Sepia,
	anchorcolor=BurntOrange,
	citecolor=MidnightBlue,
	citecolor=OliveGreen,
	filecolor=BlueViolet,
	menucolor=Yellow,
	urlcolor=OliveGreen,
}

\title[Lifting globally $F$-split surfaces]
{Lifting globally $F$-split surfaces to characteristic zero} 
\author[F. Bernasconi, I. Brivio, T. Kawakami, J. Witaszek]{Fabio Bernasconi, Iacopo Brivio, Tatsuro Kawakami and Jakub Witaszek} 

\address{Institut de Math\'ematiques, Universit\'e de Neuchâtel, 11 Rue Emile Argand, Neuchâtel, Switzerland
} 
\email{fabio.bernasconi@unine.ch}

\address{Center of Mathematical Sciences and Applications, Harvard University, Cambridge, MA, 02138}
\email{ibrivio@cmsa.fas.harvard.edu}

\address{Department of Mathematics, Graduate School of Science, Kyoto University, Kyoto 606-8502, Japan} 
\email{tatsurokawakami0@gmail.com}

\address{Princeton University, Department of Mathematics,
Fine Hall, Washington Road, Princeton NJ 08544-1000, USA} 
\email{jwitaszek@princeton.edu}

\newcommand{\Tr}[0]{{\operatorname{Tr}}}

\newcommand{\Proj}[0]{{\operatorname{Proj}}}
\newcommand{\Spec}[0]{{\operatorname{Spec}}}
\newcommand{\Hom}[0]{{\operatorname{Hom}}}

\newcommand{\Supp}[0]{{\operatorname{Supp}}}
\newcommand{\Pic}[0]{{\operatorname{Pic}}}

\newcommand{\Ex}[0]{{\operatorname{Ex}}}

\newcommand{\Aut}[0]{{\operatorname{Aut}}}

\newcommand{\Sing}[0]{{\operatorname{Sing}}}
\newcommand{\can}[0]{{\operatorname{can}}}

\newcommand{\Spf}[0]{{\operatorname{Spf}}}

\newcommand{\factor}[2]{\left. \raise 2pt\hbox{\ensuremath{#1}} \right/
	\hskip -2pt\raise -2pt\hbox{\ensuremath{#2}}}

\newcommand{\MO}{\mathcal{O}}

\newcommand{\Q}{\mathbb{Q}}
\newcommand{\Z}{\mathbb{Z}}

\newcommand{\cX}{\mathcal{X}}
\newcommand{\cY}{\mathcal{Y}}
\newcommand{\cZ}{\mathcal{Z}}
\newcommand{\cT}{\mathcal{T}}
\newcommand{\cE}{\mathcal{E}}
\newcommand{\cF}{\mathcal{F}}

\newcommand{\cL}{\mathcal{L}}
\newcommand{\cO}{\mathcal{O}}
\newcommand{\cD}{\mathcal{D}}

\newcommand{\cI}{\mathcal{I}}

\newcommand{\bZ}{\mathbb{Z}}

\usepackage{soul,xcolor}

\setstcolor{red}
\usepackage{soul,xcolor}

\setstcolor{blue}

\begin{document}
	
	\begin{abstract}
		We prove that every globally $F$-split surface admits an equisingular lifting over the ring of Witt vectors.
	\end{abstract}

	\maketitle
	\tableofcontents
		
\section{Introduction}
	
	Let $X$ be a projective variety over an algebraically closed field $k$ of characteristic $p>0$.
	Both for geometric and arithmetic purposes it is natural to ask under which conditions $X$ admits a lifting $\mathcal{X}$ to characteristic zero. The existence of such a lifting would then allow for exploiting results from complex analytic geometry (such as Hodge theory) to study the original variety in characteristic $p$. 
	
	Serre constructed examples showing there is no hope for the existence of a lifting for a general variety of positive characteristic (\cite{Ser61}). 
	Nevertheless, a general expectation is that a lifting of $X$ to characteristic zero (or at least modulo $p^2$) can often be constructed if additional hypotheses on its geometry and on the arithmetic of the Frobenius morphism $F \colon X \to X$ are satisfied. One of they key results in this direction is the following  well-known theorem.
	
	\begin{theorem}[{cf.\ \cite[Proposition 3.2]{Zda18}}] \label{thm:fsplit-lift-modp2}
	Let $X$ be a globally $F$-split scheme over a perfect field $k$ of characteristic $p>0$. Then $X$ lifts to a flat scheme $\widetilde X$ over $W_2(k)$.
	\end{theorem}
	
	In \cite{AZ21}, Achinger and Zdanowicz conjectured that every globally $F$-split smooth Calabi--Yau variety lifts to characteristic zero.  
    This is a special case of the following folklore conjecture.
	
	\begin{conjecture}[{cf.\ \cite[Section 1.7]{AZ21}}] \label{conj-F-split}
		Let $X$ be a globally $F$-split normal projective variety over an algebraically closed field $k$ of characteristic $p>0$. Then $X$ lifts to a flat projective scheme $\mathcal{X}$ over the ring $W(k)$ of Witt vectors. 
	\end{conjecture}
	
The goal of this article is to prove \autoref{conj-F-split} in dimension two. 
In fact, we show a much stronger result: that a log resolution of every globally $F$-split normal projective surface admits a lifting over the ring $W(k)$ of Witt vectors (see \autoref{t-3} for a more general statement involving pairs).

\begin{theorem}[{\autoref{t-3}}]\label{t-main-thm-lift}
Let $X$ be a globally $F$-split normal projective surface over an algebraically closed field $k$ of characteristic $p>0$. Then it is strongly log liftable, i.e. there exists a log resolution $f \colon (Y,\Ex(f)) \to X$ admitting a lifting to $\widetilde{f} \colon \mathcal{Y} \to \mathcal{X}$ over $W(k)$ such that
\begin{enumerate}
    \item[(a)] $\widetilde{f}$ is birational;
    \item[(b)] $(\mathcal{Y},\mathcal{E}x(\widetilde{f}))$ is a lifting of $(Y, \Ex(f))$ (see \autoref{def-lift});
    \item[(c)] $\mathcal{X}$ is a lifting of $X$ over $W(k)$.
\end{enumerate}
\end{theorem}
\begin{remark}
	Previous results in the literature support \autoref{conj-F-split}:
	\begin{enumerate}
		\item[(a)] globally $F$-split smooth projective varieties with trivial tangent bundle admit a canonical lifting over $W(k)$ (see \cite{Kat81} and \cite[Appendix]{MS87});
		\item[(b)] \label{b} globally $F$-split (equivalently, ordinary) K3 and Enriques surfaces admit a canonical lifting over $W(k)$ (\cite{Del81, Nyg83, LT19}).
	\end{enumerate}
\end{remark}

In the past few years, several authors investigated liftability of log resolutions of klt del Pezzo surfaces over $W(k)$, especially for its connections with Kodaira-type vanishing theorems (\cite{CTW17, ABL20, KN20, Nag21, Lacini24}). 
Showing that a log resolution of a variety $X$ lifts to characteristic zero is much more impactful than showing that $X$ itself lifts, as it permits to compare the singularities of a variety with those of the lifting in characteristic zero (see e.g.\ \autoref{p-good-lifting}). In particular, as a corollary to \autoref{t-main-thm-lift} we can construct liftings of globally $F$-split surfaces over $W(k)$ preserving the type of singularities and the Picard rank.

	\begin{corollary}\label{c-1}
		Let $k$ be an algebraically closed field of characteristic $p>0$.
		Let $X$ be a normal projective  globally $F$-split surface over $k$.
		Then there exists a lifting $\mathcal{X}$ of $X$ over $W(k)$ with  geometric generic fibre $\mathcal{X}_{\overline{K}}$ such that the following holds:
		\begin{enumerate}
			\item[(a)] there is a natural bijection  of sets $g \colon \Sing(\mathcal{X}_{\overline{K}}) \xrightarrow{\sim} \Sing(X)$;
			\item[(b)] if $x \in \Sing(\mathcal{X}_{\overline{K}})$, then the weighted dual graph of the minimal resolution (see \autoref{def-weighted-dual-graph}) at $x$ is equal to the one at $g(x)$;
			\item[(c)] if $X$ has rational (resp. klt) singularities, then $\mathcal{X}$ has rational (resp. klt) singularities;
			\item[(d)] if $X$ has rational singularities, then $\rho(X)=\rho(\mathcal{X}_{\overline{K}})$.
		\end{enumerate}
	\end{corollary}
	
In what follows, we explain some of the consequences of our results. 
For example, we can prove a bound on the Gorenstein index of globally $F$-split klt Calabi--Yau surfaces which is independent of the characteristic.
\begin{corollary}\label{c-index-CY}
Let $k$ be an algebraically closed field of characteristic $p>0$.
		Let $X$ be a globally $F$-split klt projective surface  over k such that $K_X\equiv0$.
		Then 
		the Gorenstein index of $X$ and the global index of $K_X$ are at most $21$. In particular, $X$ is $\frac{1}{21}$-log canonical.
\end{corollary}	
\noindent We refer to \autoref{def:definition-of-global-index} for the definitions of the Gorenstein and the global indices. 

As a further application we can show the Bogomolov bound on the number of singular points of globally $F$-split klt del Pezzo surfaces in characteristic $p$  (see \cite{KM99, Bel09, LX21} for the bounds in characteristic zero). 
	
\begin{corollary}\label{c-bogomolovbound}
	Let $k$ be an algebraically closed field of characteristic $p>0$.
	Let $X$ be a globally $F$-split klt del Pezzo surface over $k$.
	Then $X$ has at most $2\rho(X)+2$ singular points.
\end{corollary}
	
	\begin{remark}
		Thanks to the $F$-split condition, we avoid to use the explicit classification of \cite{Lacini24}.
		Moreover, we are also able to include the case of low characteristic, giving a complete answer. 
		Note that the examples constructed in \cite{CT19,Ber21, Lacini24} in characteristic $p\in \left\{2,3,5\right\}$ show that \autoref{t-main-thm-lift}, \autoref{c-1} and \autoref{c-bogomolovbound} fail for non-globally-$F$-split klt del Pezzo surfaces in low characteristic.
	\end{remark}

\begin{remark}	
The third author recently showed that there exists $p_0>0$ such that log Calabi--Yau surface pairs are log liftable over $W(k)$ if $p >p_0$ (\cite[Theorem 1.3]{Kaw21}). 
At the moment, an explicit bound on $p_0$ is not known and it is not known if general log Calabi--Yau surfaces are strongly liftable.
\end{remark}
  
\textbf{Sketch of the proofs.}
The proof of \autoref{t-main-thm-lift} consists of two parts:  
\begin{enumerate}
    \item[(a)] showing that $X$ is log liftable over $W(k)$, and then 
    \item[(b)] proving that such a lifting descends to $X$. 
\end{enumerate}
Recall that $X$ is \emph{log liftable} if $(Y, \Ex(f))$ admits a lifting over $W(k)$, where $f \colon Y \to X$ is a log resolution (see \autoref{def-liftability}).

Note that Part (b) is easy when $X$ has klt, and so rational, singularities by standard deformation theoretic arguments (cf.\ \autoref{t-cynkvanstraten}) but it is much more difficult in general as $F$-splitness only implies that $X$ has log canonical singularities.\\

\textit{Sketch of the proof of Part (a).} 
	Since $X$ is globally $F$-split we know that $X$ has log canonical singularities and $-K_X$ is $\Q$-effective by \autoref{l-Q-K+delta}. We thus distinguish three cases:
	
	\begin{enumerate}
	    \item[(i)] $K_X \sim_{\mathbb{Q}} 0$ and $X$ is klt,
	    \item[(ii)] $\kappa(K_X)=-\infty$ and $X$ is klt,
	    \item[(iii)] $X$ is strictly log canonical.
	\end{enumerate}
	
	Case  (i) of $K$-trivial varieties with klt singularities is discussed in \autoref{s-K0-klt}.  
	By taking the canonical covering and a careful study of lifts of group actions we can reduce to the case when $X$ has klt Gorenstein (hence canonical) singularities. 
	We conclude (i) by going through the Enriques-Kodaira classification of Bombieri and Mumford (\cite{BM76,BM77}) and  applying the theory of canonical liftings of $K$-trivial smooth ordinary $K$-trivial surfaces as developed in \cite{Nyg83, MS87, LT19} (see \autoref{prop:lift_minres_K0}).
	
	In Case (ii), let $\phi \colon X \to W$ be an output of a two-dimensional $K_X$-Minimal Model Program (MMP), see \cite{Tan18}. It is easy to see that it is enough to show that $W$ is log liftable (\autoref{l-loglifting-image}). Since $\kappa(K_X)=-\infty$, the variety $W$ is either a klt del Pezzo surface of Picard rank one, or it admits a Mori fibre space structure over a curve. In the latter case we follow the ideas from \cite{Kaw21} to prove liftability (\autoref{LCTF2}). In the former case where $W$ is a del Pezzo surface we argue as follows in  \autoref{l_del_pezzo_case}. 
	Let $g \colon Y \to W$ be the minimal resolution of $W$. Then $Y$ is also globally $F$-split (\autoref{l-GFR-pullback}), and so $Y$ lifts modulo $p^2$ as explained in \autoref{thm:fsplit-lift-modp2}. In fact, more is known, the whole pair $(Y, E)$ for $E=\Ex(g)$ lifts modulo $p^2$ as indicated by the following simple but somewhat very surprising result:	
	
	\begin{lemma}[\textup{\cite[Lemma 5.2.2]{AZ21}}] \label{p-liftpair-W2}
		 Let $Y$ be a smooth projective globally $F$-split variety over a perfect field $k$ and let $E$ be a reduced simple normal crossing divisor. 
		 Then $(Y,E)$ admits a lifting to $W_2(k)$.
	\end{lemma}

	 Therefore, we can invoke the logarithmic variant of the  theorem of Deligne-Illusie (\cite{DI87,Ha98}), and so apply Akizuki-Nakano vanishing to show that $H^2(Y, T_Y(-\,\mathrm{log}\, E))=0$. Therefore, $(Y,E)$ lifts over the ring of Witt vectors by deformation theory.
	
	Finally, Case (iii) follows by a similar argument to (ii) but we first take a dlt blow-up $h \colon Z \to X$ and run a $K_Z$-MMP. Here, $\kappa(K_Z) = -\infty$ as $X$ is not klt.\\
	
	\textit{Sketch of the proof of Part (b).} 
	We now describe the strategy of the proof of liftability of globally $F$-split surface pairs.
    As a first step in \autoref{c-lift-min-dlt} we use log liftability to show the existence of a lifting for a dlt modification $(Y, D_Y)$ of $(X,D)$.
    In \autoref{p-lifting-criterion}, we give a sufficient criterion to descend liftability from $Y$ to $X$ in terms of extension of line bundles, which turns to be easily verified in the case where $H^0(X,\mathcal{O}_X(K_X))=H^2(X, \mathcal{O}_X)=0$ (\autoref{p-lc-nottriv}). 
    
    We thus reduce to the case of globally $F$-split surfaces $X$ with $K_X \sim 0$ and strictly log canonical singularities. 
    In \autoref{l-logcanonical-CY} we give a crepant birational classification of $X$ into 3 classes of snc Calabi--Yau pairs $(Z,D_Z)$.
    For each of these pairs, we construct a `canonical' lifting $(\mathcal{Z}_\can, \mathcal{D}_\can)$ over $W(k)$ and in \autoref{p-strictlyLC-case} we use the special properties of $\mathcal{Z}_\can$ to show the existence of a lifting of $X$ over $W(k)$.\\

\textbf{Acknowledgments.}
	The authors thank A.~Petracci, F.~Carocci, P.~Cascini, C.D.~Hacon, G.~Martin, L.~Stigant, R.~Svaldi, S.~Yoshikawa, T.~Takamatsu, M.~Nagaoka, and M.~Zdanowicz for useful discussions and comments on the content of this article. 
 The authors are also grateful to the referee for reading the manuscript very carefully and providing many valuable comments that improved our paper.

	\section{Preliminaries}
	\subsection{Notation} \label{ss-notation}
	\begin{enumerate}
		\item[(a)] Throughout this article, unless stated otherwise, $k$ denotes an algebraically closed field of prime characteristic $p>0$.
		\item[(b)] We denote by $W(k)$ the \textit{ring of Witt vectors of $k$}. As $k$ is perfect, it is a complete discrete valuation ring (DVR) of mixed characteristic $(0,p)$ with maximal ideal $\mathfrak{m}=(p)$ and residue field $k = W(k)/(p)$. We denote by $W_m(k)=W(k)/(p^m W(k))$ the \emph{ring of Witt vectors of length $m$} and by $K$ the field of fractions of $W(k)$.
		\item[(c)] Let $X$ be an $\mathbb{F}_p$-scheme. We denote by $F \colon X \to X$ the absolute \emph{Frobenius morphism} and, for each $e>0$, we denote by $F^e$ the $e$-th iterated power of absolute Frobenius. We say  that $X$ is $F$-finite if $F$ is a finite morphism.
		\item[(d)] We say that $X$ is a \emph{variety} if it is an integral scheme which is separated and of finite type over a field or a complete DVR.
		We say that $X$ is a curve, resp.~ a surface, if $X$ is a variety over a field of dimension one, resp.~ two.
		\item[(e)] We say that $(X,\Delta)$ is a \emph{pair} if $X$ is a normal variety and $\Delta$ is an effective $\mathbb{Q}$-divisor. If $K_X+\Delta$ is $\Q$-Cartier, we say that $(X,\Delta)$ is a \emph{log pair}. If $\Delta$ is not necessarily effective, then we say  that $(X,\Delta)$ is a \textit{sub (log) pair}.
		\item[(f)] Given a pair $(X, \Delta)$, we say that $f \colon (Y,E) \to X$ is a \emph{log resolution} of $(X,\Delta)$ if $f$ is a proper birational morphism, the exceptional locus $ E\coloneqq \Ex(f)$ is of pure codimension one, and $(Y, \Supp(f_*^{-1}\Delta)+ E)$ is an snc pair, where $f_*^{-1}\Delta$ denotes the proper transform of $\Delta$ on $Y$.  By abuse of notation we shall sometimes drop $E$ and call $f \colon Y \to X$ a log resolution.
		\item[(g)] For the definition of the singularities of pairs appearing in the MMP (such as \emph{canonical, klt, dlt, log canonical}) we refer to \cite[Definition 2.8]{kk-singbook}.
		\item[(h)]  We say $f \colon (Y,\Delta_Y) \to (X,\Delta)$ is a proper \emph{birational morphism of pairs}, if $(Y,\Delta_Y)$ and $(X,\Delta)$ are pairs, $f \colon Y \to X$ is a proper birational morphism, and $f_*\Delta_Y = \Delta_X$.
		\item[(i)]\label{def-crepant} Let $f\colon (Y, \Delta_Y)\to (X, \Delta_X)$ be a proper birational morphism of log pairs. 
		We say  that $f$ is \textit{crepant} if $K_Y+ \Delta_Y=f^*(K_X+\Delta_X)$. 
		More generally, the pairs $(Y, \Delta_Y)$ and $(X, \Delta_X)$ are said to be \emph{crepant birational} if there exist a sub log pair $(Z, \Delta_Z)$ and crepant proper birational morphisms $p\colon (Z, \Delta_Z)\to (Y,\Delta_Y)$ and $q\colon (Z, \Delta_Z)\to (X, \Delta_X)$.
		\item[(j)] Let $(X, \Delta)$ be a log pair. We say that $(X,\Delta)$ is a \emph{log Calabi--Yau pair} (resp.~ a \emph{log Fano pair}) if it has log canonical singularities and $K_X+\Delta \sim_{\mathbb{Q}} 0$ (resp. it has klt singularities and $-(K_X+\Delta)$ is ample). 
		We say  that $X$ is a variety \emph{of Calabi--Yau} type (resp.~ \emph{Fano type}) if there exists a  $\mathbb{Q}$-divisor $\Delta$ such that $(X,\Delta)$ is a log Calabi--Yau (resp.~ log Fano) pair. For historical reasons, a Fano (type) surface is called a \emph{del Pezzo} (type) surface.
		\item[(k)] If $f \colon Y \to X$ is a finite \'etale  morphism of schemes,  then we write $\Aut_X(Y)$ for the automorphism group of $Y$ over $X$  acting on the right on $Y$. 
		We say  that $f$ is \emph{Galois} if $\Aut_X(Y)$ acts transitively on the geometric fibres of $f$. 
		If $X$ and $Y$ are normal, then $f$ is Galois if and only if the field extension $K(Y)/K(X)$ is Galois.
		\item[(l)]  A morphism $\pi \colon X \to Y$ of normal varieties is called \emph{a quasi-\'etale covering} if it is a finite surjective morphism which is \'etale over the codimension one points of $Y$. 
		If $f$ is quasi-\'etale,  then we say that it is Galois if the field extension $K(Y)/K(X)$ is Galois.
		\item[(m)]   Given a normal proper variety $X$ over any field $k$, we denote by $\rho(X)$ the Picard rank of $X$. For a $\mathbb{Q}$-Cartier divisor $D$, we denote by $\kappa(D)$ its Iitaka dimension.
		\item[(n)] Given a pair $(X, D)$ where $D$ is a reduced Weil divisor, we say a group scheme $G$ acts on $(X,D)$ if $G$ acts on $X$ and its actions preserves the open set $U\coloneqq X \setminus D$.
		\item[(o)] Given a normal variety $X$ over any field $k$ and a reduced Weil divisor $D=\sum_i D_i$, we denote by $\Omega_X^{[q]}(\log D):=j_*\Omega_{U/k}^{q}(\log D|_U)$ the sheaf of \emph{reflexive logarithmic differential $q$-forms} where $U$ is the snc locus of $(X,D)$ and $j\colon U\hookrightarrow X$ is the natural inclusion.
		We denote by $T_X(-\log D):=(\Omega_X^{[1]}(\log D))^*$ the logarithmic tangent sheaf of $(X,D)$.
	\end{enumerate}
	
	\subsection{Frobenius splitting}
	
	We first recall the notion of Frobenius splitting (in short, $F$-splitting) for $\mathbb{F}_p$-schemes.	
	
	\begin{definition}
		Let $X$ be a normal $F$-finite $\mathbb{F}_p$-scheme and let $\Delta$ be an effective  $\mathbb{Q}$-divisor on $X$.
		We say  that the pair $(X, \Delta)$ is \emph{globally sharply $F$-split} if there exists $e \in \mathbb{N}$ for which the natural composition map
		\[\mathcal{O}_X \to F^e_* \mathcal{O}_X \hookrightarrow  F^e_* \mathcal{O}_X(\lceil{ (p^e-1)\Delta \rceil})  \]
		splits in the category of $\mathcal{O}_X$-modules.
		When $\Delta$ is integral (or more generally when $m\Delta$ is integral for $m \in \bZ_{>0}$ which is not divisible by $p$), we will simply say that $(X,\Delta)$ is globally $F$-split.
	\end{definition}
	
	Globally $F$-split varieties should be thought of as varieties of Calabi--Yau type whose arithmetic is well-behaved.
	
	\begin{proposition}\label{l-Q-K+delta}
		Let $k$ be an $F$-finite field and let $(X,\Delta)$ be a globally sharply $F$-split quasi-projective normal variety over $k$. 
		Then
		\begin{enumerate}
			\item[(a)] there exists a $\mathbb{Q}$-divisor $\Gamma \geq 0$ such that $(X, \Delta +\Gamma)$ is a globally sharply $F$-split log Calabi--Yau pair and $(p^e-1)(K_X+\Delta+\Gamma) \sim 0$ for some $e>0$;
			\item[(b)] if $\dim X=2$, then $(X,\Delta)$ has log canonical singularities.
		\end{enumerate}
	\end{proposition}
	\begin{proof}
		By  \cite[Theorem 4.3]{SS10}, there exists a $\mathbb{Q}$-divisor $\Gamma \geq 0$ such that $(X,\Delta +\Gamma)$ is a globally $F$-split log pair and $K_X+\Delta+\Gamma \sim_{\mathbb{Q}} 0$.  By \cite[Theorem 3.3]{HW02}, $(X, \Delta+\Gamma)$ has log canonical singularities.
		
		To prove (b) it is sufficient to prove that $K_X+\Delta$ is $\mathbb{Q}$-Cartier. 
		We fix $x \in X$ and we divide the proof  into two cases.
		Suppose $\mathcal{O}_{X,x}$ is a germ of a rational surface singularity. Then it is $\mathbb{Q}$-factorial by \cite[Proposition B.2]{Tan15}.
		If $x$ is not a rational singularity, then $x \notin \Supp(\Delta+\Gamma)$ by \cite[Proposition 2.28]{kk-singbook}. In particular, $K_X$ is $\mathbb{Q}$-Cartier in a neighbourhood of $x$ and then $(X,\Delta)$ is log canonical at $x$.
	\end{proof}
	
	We collect some well-known properties on the behaviour of globally sharply $F$-split pairs under birational operations and quasi-\'etale morphism.
	
	\begin{lemma}\label{l-gfs-image}
		Let $k$ be an $F$-finite field.
		Let $(Y, \Gamma)$ be a globally sharply $F$-split pair  over $k$ and let $f \colon Y \to X$ be a proper birational  morphism between normal varieties.
		Then $(X, \Delta:=f_*\Gamma)$ is globally sharply $F$-split.
	\end{lemma}
	\begin{proof}
	Let $i \colon U \hookrightarrow X$ be an open subset such that $f^{-1}(U) \xrightarrow{f} U$ is an isomorphism and $\mathrm{codim}_X\, (X \,\backslash\, U) \geq 2$. 
	Set $V \coloneqq f^{-1}(U)$ and pick $e>0$ such that the map \[(F^e)^* \colon \mathcal{O}_Y \to F^e_* \mathcal{O}_Y(\lceil{ (p^e-1)\Gamma \rceil})\] splits. Then 
	\[
	(F^e)^* \colon \mathcal{O}_V \to F^e_* \mathcal{O}_V(\lceil{ (p^e-1)\Gamma \rceil|_V}),
	\]
	splits as well, and thus so does
	\[
	(F^e)^* \colon \mathcal{O}_U \to F^e_* \mathcal{O}_U(\lceil{ (p^e-1)\Delta \rceil|_U}).
	\]
	Since 
	\[
	\mathcal{H}\textrm{om}_{\mathcal{O}_X}(F^e_* \mathcal{O}_X(\lceil{ (p^e-1)\Delta \rceil}), \mathcal{O}_X) \cong i_* \mathcal{H}\textrm{om}_{\mathcal{O}_U}(F^e_* \mathcal{O}_U(\lceil{ (p^e-1)\Delta \rceil|_U}), \mathcal{O}_U)
	\]
	by normality of $X$, we get that $(F^e)^* \colon \mathcal{O}_X \to F^e_* \mathcal{O}_X(\lceil{ (p^e-1)\Delta \rceil})$ splits, and so $(X,\Delta)$ is globally sharply $F$-split. 
	\end{proof}

	Being globally $F$-split is stable for crepant morphisms of log pairs (for the definition of crepant morphism, see \autoref{ss-notation} (i)).
	
	\begin{lemma}[{\cite[Lemma 3.3]{GT16}}]\label{l-GFR-pullback}
		Let $k$ be an $F$-finite field.
		Let $(X, \Delta)$ be a globally sharply $F$-split log pair over $k$. 
		Let $f \colon (Y, \Delta_Y) \to (X, \Delta)$ be a crepant proper birational morphism of log pairs.
		Then $(Y, \Delta_Y)$ is globally sharply $F$-split.
	\end{lemma}

We remark that in \autoref{l-GFR-pullback} it is crucial that the boundary divisor $\Delta_Y$ is effective.
	
	Being globally $F$-split is stable under the passage to quasi-\'etale covers.
	
	\begin{lemma}[{\cite[Lemma 11.1]{PZ19}}] \label{l-F-split-quasietale}
		Let $k$ be an $F$-finite field. Let $(X,\Delta)$ be a pair  over $k$ and let $\pi \colon Y \to  X$ be a quasi-\'etale $k$-morphism between normal $k$-varieties.
		If  $(X,\Delta)$ is globally sharply $F$-split, then $(Y, \Delta_Y:=\pi^*\Delta)$ is globally sharply $F$-split.
	\end{lemma}

\subsection{Log liftability over $W(k)$}
	
We fix $k$ to be an algebraically closed field of characteristic $p>0$. We recall the notion of liftability for pairs (cf.\ \cite{EV92}).
We identify a prime divisor $D$ on $X$ with its naturally associated reduced subvariety of codimension one contained in $X$.
	
\begin{definition}\label{def-lift}
	Let $(X, D=\sum_{i=1}^r D_i)$ be a pair over $k$ where $D_1,...,D_r$ are distinct prime divisors. A \emph{lifting} of $(X,D)$ over a scheme $T$ consists of
	\begin{enumerate}
		\item[(a)] a flat and separated morphism $\mathcal{X} \to T$;
		\item[(b)] closed subschemes $\mathcal{D}_i \subset \mathcal{X}$, flat over $T$ for $i=1, \dots, r$;
		\item[(c)] a morphism $\alpha \colon \Spec(k) \to T$ and an isomorphism $\gamma \colon \mathcal{X} \times_T \Spec(k) \xrightarrow{\cong} X$ such that $\gamma (\mathcal{D}_i\times_T \Spec(k)) = D_i$ for every $i=1, \dots, r$.
		\end{enumerate}
	By abuse of notation, we often identify $(\mathcal{X}, \mathcal{D}) \times_T \Spec(k)$ with $(X,D)$.
	If $T=\Spec(W(k))$, we say that $(\mathcal{X}, \mathcal{D})$ is a lifting of $(X,D)$ over the ring  $W(k)$ of Witt vectors.
\end{definition}

\begin{definition}
    Let $(R, \mathfrak{m})$ be a complete DVR, and denote by $\Spf(R)$ the formal completion of $\Spec(R)$ at $\mathfrak{m}$.
    Let $(X, D= \sum D_i)$ be a pair over $k = R/\mathfrak{m}$.
    A \emph{formal lifting} $(\mathfrak{X}, \mathfrak{D})$ of  $(X,D)$ over $\Spf(R)$ consists of 
    \begin{enumerate}
        \item[(a)] a formal scheme $\mathfrak{X} \to \Spf(R)$;
        \item[(b)] formal subschemes $\mathfrak{D}_i\subset \mathfrak{X}$ for any $i>0$; 
        \item[(c)] for any $n>0$, the truncation $(X_n, D_n) = (\mathfrak{X}, \mathfrak{D}) \times_{\Spf(R)} \Spec(R/\mathfrak{m}^{n+1})$ is a lifting of $(X, D)$ over $\Spec(R/\mathfrak{m}^{n+1})$.
    \end{enumerate}
    If $R=W(k)$, we say that $(\mathfrak{X}, \mathfrak{D})$ is a formal lifting of $(X,D)$ over the ring of Witt vectors.
\end{definition}

\begin{definition} \label{def-lift-morphism}
   Let $f \colon (Y, D_Y) \to (X, D_X=\sum_i D_i)$ be a proper birational morphism of pairs over $k$, where $D_X$ is a reduced Weil divisor and $D_Y=E+f_*^{-1}D_X$, where $E$ is $f$-exceptional.
   A \emph{lifting} $\widetilde{f}$ of $f$ over a scheme $T$ consists of 
    \begin{enumerate}
        \item[(a)] a lifting $(\mathcal{Y}, \mathcal{D}_{\mathcal{Y}})$ of $(Y, D_Y)$, and a lifting $(\mathcal{X}, \mathcal{D}_{\mathcal{X}})$ of $(X, D_X)$ over $T$;
        \item[(b)] a proper morphism $\widetilde{f} \colon (\mathcal{Y}, \mathcal{D}_Y) \to (\mathcal{X}, \mathcal{D}_X)$ such that $\widetilde{f}_*\mathcal{O}_{\mathcal{Y}}=\mathcal{O}_{\mathcal{X}}$ and the base change $\widetilde{f} \times_{T} \Spec(k) \colon Y \to X$ coincides with $f$.
    \end{enumerate}
	If $T=\Spec(W(k))$, we say that $\widetilde{f}$ is a lifting of $f$ over the ring $W(k)$ of Witt vectors.
\end{definition}
	
The following guarantees that a lifting for an snc pair as in \autoref{def-lift} is locally snc over a regular base.
	
	\begin{lemma}\label{l-loc-triv}
		Let $(X, D=\sum_{i=1}^r D_i)$ be an snc proper pair over $k$ and let $(\mathcal{X}, \mathcal{D})$ be a lifting over a regular local scheme $T$. Then $(\mathcal{X}, \mathcal{D})$ is relatively snc over $T$. In particular, if $\bigcap_{j \in J} \mathcal{D}_j$ is not empty, then it is a smooth $T$-scheme of relative dimension $\dim(X)-|J|$.
	\end{lemma}
	\begin{proof}
		See \cite[Remark 2.7]{Kaw21}.
	\end{proof}
	
The following is a flatness criterion for Cartier divisors we will be repeatedly using.

\begin{lemma}\label{l-flat-Cartier}
    Let $X$ be a proper variety over $k$ and $D$ an effective Cartier divisor on $X$.
    Let $f \colon \mathcal{X} \to T$ be a flat lifting of $X$ over a local scheme $T$.
    If $\mathcal{D}$ is an effective Cartier divisor on $\mathcal{X}$ such that $\mathcal{D}|_X=D$, then $\mathcal{D} \to T$ is flat.
\end{lemma}

\begin{proof}
We can suppose $X$ and $\mathcal{X}$ are affine schemes and we let $f$ and $\widetilde{f}$ be local equations defining $D$ and $\mathcal{D}$.
By considering the short exact sequence 
\[0 \to \mathcal{O}_\mathcal{X} \xrightarrow{\widetilde{f}} \mathcal{O}_\mathcal{X} \to \mathcal{O}_\mathcal{D} \to 0,\]
we obtain the following exact sequence:
\[\text{Tor}_1^{R}(k, \mathcal{O}_{\mathcal{X}})  \to \text{Tor}_1^{R}(k, \mathcal{O}_{\mathcal{D}}) \to \mathcal{O}_X \xrightarrow{f} \mathcal{O}_X \to \mathcal{O}_D \to 0.\]
As $f$ is a non-zero divisor, $\mathcal{O}_X \xrightarrow{f} \mathcal{O}_X$ is injective and thus we deduce that \[\text{Tor}_1^{R}(k, \mathcal{O}_{\mathcal{X}})  \to \text{Tor}_1^{R}(k, \mathcal{O}_{\mathcal{D}})\]
is surjective.
Since $\text{Tor}_1^{R}(k, \mathcal{O}_{\mathcal{X}})=0$, so is $\text{Tor}_1^{R}(k, \mathcal{O}_{\mathcal{D}})$. By applying the local criterion for flatness 
\cite[Lemma 2.1]{Har2}, we conclude $\mathcal{D}$ is flat over $T$.
\end{proof}	
	
	We recall the fundamental notion of log liftability over the Witt vectors for singular varieties that we will use in this article. 
	
	\begin{definition}\label{def-liftability}
		Let $(X,D)$ be a pair over $k$, where $D$ is a reduced Weil divisor. We say that $(X,D)$ is \emph{log liftable over the ring $W(k)$ of Witt vectors} if there exists a log resolution $f \colon (Y,E) \to X$ of $(X,D)$ such that the snc pair $(Y,E+f^{-1}_{\ast}D)$ admits a lifting over $W(k)$.
		
		We say it is \emph{strongly log liftable over $W(k)$} if the proper birational morphism of pairs $f \colon (Y, E+f^{-1}_{\ast}D) \to (X,D)$ lifts over $W(k)$.
	\end{definition}
	
	We stress that in the definition of log liftability, we do not require that the morphism $f$ lifts.
	The following shows that log liftability is a well-behaved notion in the case of surfaces. The existence of log resolutions for excellent surfaces is proven in \cite{Lip78}.

    \begin{lemma}[{cf.\ \cite[Lemma 2.7]{KN20}}]\label{l-loglift-surf} 
		Let $(X, D)$ be a normal surface pair over $k$, where $D$ is a divisor.
		Then the following are equivalent:
		\begin{enumerate}
			\item[(a)] for some log resolution $f \colon  (Y,E) \to X$  of $(X,D)$, the pair $(Y,f_*^{-1}D+E)$ admits a formal lifting over $W(k)$;
			\item[(b)] for all log resolutions $f \colon (Y,E) \to X$ of $(X,D)$, the pair $(Y,f_*^{-1}D+E)$ admits a formal lifting over $W(k)$.
		\end{enumerate} 
		Moreover if $H^2(Y, \mathcal{O}_Y)=0$ for some resolution $Y \to X$, then any formal lifting over $W(k)$  of a resolution $Z \to X$ is algebraisable, in particular $(X,D)$ is log liftable.
		Finally, if $X$ has klt singularities it is sufficient to check the liftability of the minimal resolution of $X$.
	\end{lemma}
	
	\begin{proof}
		(b) $\Rightarrow $ (a) is obvious. 
		We now show (a) $\Rightarrow$ (b).
		Suppose that there exists a log resolution $(Y, f_{*}^{-1}D+E)$ lifting over $W(k)$  and let $g \colon Z \to X$ be another log resolution of  $(X,D)$.
		By a resolution of indeterminacies of rational maps between surfaces, there exists a finite number of blow ups $h \colon  W \to Y$ at smooth points of $Y$ such that $\pi \colon W \to X$ is a log resolution of $(X,D)$ and there exists a birational morphism $W \to Z$. By \cite[Proposition 2.9]{ABL20} the pair $(W, \pi_{*}^{-1}D+\Ex(\pi))$ lifts over $W(k)$. 
		Finally by applying \cite[Proposition 4.3]{AZ-nonlift} $(Z, g_*^{-1}D+\Ex(g))$ lifts over $W(k)$.
		
		If $H^2(Y, \MO_Y)=0$ for some log resolution, then also $H^2(Z, \mathcal{O}_Z)=0$ and any formal lifting over $W(k)$ of $Z$ is algebraisable by \cite[Corollary 8.5.6 and Corollary 8.4.5]{FAG}.
		
		As the minimal resolution of a klt singularity is a log resolution by their classification \cite[Corollary 3.31]{kk-singbook}, the last statement is clear.
	\end{proof}
	
	The following is a useful remark on the log liftability of surface pairs which we will use repeatedly.
	
	\begin{lemma}\label{l-loglifting-image}
		Let $\pi \colon Y \to X$ be a proper birational morphism of projective normal surfaces over $k$ and let $D$ be a reduced Weil divisor on $Y$.
		If $(X, \pi_*D)$ is log liftable over $W(k)$, then so is $(Y, D)$.
	\end{lemma}
	
	\begin{proof}
		Take a log resolution $f \colon Z \to X$ of $(X, \pi_*D)$ such that $(Z, f_*^{-1}(\pi_*D)+\Ex(f))$ lifts over $W(k)$. By passing to a higher model and by \cite[Proposition 2.9]{ABL20} we can assume that $f \colon Z \to  X$ admits a factorisation  $g \colon Z \to Y$.
		Since $ f_*^{-1}(\pi_*D) +\Ex(f) \supset g_*^{-1} D +\Ex(g),$ we conclude that also $(Y, D)$ is log liftable over $W(k)$.
	\end{proof}
	
	\begin{remark}
		Note that \autoref{l-loglift-surf} and \autoref{l-loglifting-image} are specific to surfaces and they do not extend to higher dimensions as shown by the examples of \cite[Theorem 2.4]{LS14}. 
	\end{remark}

	\subsection{Deformation theory toolbox}\label{s-def-toolbox}
	
	In this section we collect results on deformation theory we will need throughout the article.
	
	\subsubsection{Descent of liftings under contractions}
	
	The following  result provides a sufficient cohomological criterion for the existence of a lifting for a contraction (see \cite{AZ-nonlift, CvS09}).

	\begin{theorem}\label{t-cynkvanstraten}
		Let $\Spec(A) \to \Spec(A')$ be a closed immersion of local Artinian schemes defined by a principal ideal $J=(\pi)$ of square zero.
		Let $f \colon Y \to X$  be a morphism of flat $A$-schemes.
		Let $\left\{E_i \right\}_{i \in I}$ (resp.~ $\left\{F_i\right\}_{i \in I}$) be a collection of closed subsets of $Y$ (resp.~ of $X$).
		Assume that 
		\begin{enumerate}
			\item[(a)] $f_*\mathcal{O}_Y=\mathcal{O}_X$ and $R^1f_*\mathcal{O}_Y=0$;
			\item[(b)] $f_*\mathcal{O}_{E_i}=\mathcal{O}_{F_i}$ and $R^1f_*\mathcal{O}_{E_i}=0$ for each $i \in I$.
		\end{enumerate}
		Let $(\mathcal{Y}, \left\{\mathcal{E}_i \right\}_{i \in I})$ be a lifting of $(Y,\left\{E_i \right\}_{i \in I})$ over $A'$. Then
		\begin{enumerate}
			\item[(1)] there exists a natural lifting $(\mathcal{X}, \left\{\mathcal{F}_i\right\}_{i \in I})$ of $(X, \left\{F_i\right\}_{i\in I})$ together with a lifting \[\tilde{f} \colon (\mathcal{Y}, \left\{\mathcal{E}_i \right\}_{i \in I}) \to (\mathcal{X}, \left\{\mathcal{F}_i\right\}_{i \in I})\] of $f$ over $A'$;
			\item[(2)] $\widetilde{f}_*\mathcal{O}_{\mathcal{Y}}= \mathcal{O}_{\mathcal{X}}$ and $R^1\widetilde{f}_*\mathcal{O}_{\mathcal{Y}}=0$;
			\item[(3)] $\tilde{f}_*\mathcal{O}_{\mathcal{E}_i}=\mathcal{O}_{\mathcal{F}_i}$ and $R^1\tilde{f}_*\mathcal{O}_{\mathcal{E}_i}=0$ for each $i \in I$.  
		\end{enumerate} 
	\end{theorem}
	
	\begin{proof}
		As topological spaces, we  set $\mathcal{X}_{\text{top}}\coloneqq X$ and $\widetilde{f}_{\text{top}}:=f$. 
		We define the sheaf of rings on $\mathcal{X}$ as follows: \[\mathcal{O}_{\mathcal{X}}:=f_*\mathcal{O}_{\mathcal{Y}}.\]
		We must verify that the scheme $\mathcal{X}$ is flat over $A'$.
		As $\mathcal{Y}$ is flat over $A'$, there is a short exact sequence  of sheaves of abelian groups:
		\[0 \to \mathcal{O}_Y \to \mathcal{O}_{\mathcal{Y}} \to \mathcal{O}_Y \to 0.\]
		By considering the push-forward via $\widetilde{f}$ we conclude that the sequence 
		\[0 \to \mathcal{O}_X \to \mathcal{O}_{\mathcal{X}} \to \mathcal{O}_X \to 0\]
		is exact since $R^1f_* \mathcal{O}_Y=0$ by hypothesis.
		Therefore $\mathcal{X}$ is flat over $A'$ and $R^1\widetilde{f}_* \mathcal{O}_{\mathcal{Y}}=0$.
		
		We apply the same construction to construct the liftings $\mathcal{F}_i$ of $F_i$. 
		We are only left to verify that $\mathcal{F}_i$ is a subscheme of $\mathcal{X}$. 
		As 	$0 \to \mathcal{I}_{\mathcal{E}_i} \to \mathcal{O}_\mathcal{Y} \to \mathcal{O}_{\mathcal{E}_i} \to 0$ is exact we conclude that 
		\[ \widetilde{f}_*\mathcal{O}_{\mathcal{Y}}= \mathcal{O}_{\mathcal{X}} \twoheadrightarrow \widetilde{f}_*\mathcal{O}_{\mathcal{E}_i}=\mathcal{O}_{\mathcal{F}_i}\] provided that $R^1\widetilde{f}_*\mathcal{I}_{\mathcal{E}_i}$ vanishes.
		Note that $R^1f_*\mathcal{I}_{E_i}$ vanishes because it fits in the short exact sequence $\mathcal{O}_X \twoheadrightarrow \mathcal{O}_{F_i} \to R^1f_* \cI_{E_i} \to R^1f_*\mathcal{O}_Y=0 $.
		Consider the sequence (which exists by the snake lemma):
		\[0 \to \mathcal{I}_{\mathcal{E}_i}(-Y) \to  \mathcal{I}_{\mathcal{E}_i} \to \mathcal{I}_{E_i} \to 0.\]
		By applying the push-forward, the projection formula and the equality $R^1f_*\mathcal{I}_{E_i}=0$ we deduce the  surjectivity of
		$
		R^1\widetilde{f}_*\mathcal{I}_{\mathcal{E}_i} \otimes \mathcal{O}_{\mathcal{X}}(-X) \twoheadrightarrow R^1\widetilde{f}_*\mathcal{I}_{\mathcal{E}_i}$.
		As $J$ is nilpotent, we conclude that $R^1\widetilde{f}_*\mathcal{I}_{\mathcal{E}_i}=0$.
	\end{proof}

	\subsubsection{Deformations of line bundles}
	
	We study the deformation theory of line bundles equipped with a trivialisation on a closed subscheme.  
	This theory follows closely the classical one presented in \cite[Section 8.5.2]{FAG}.
	
	\begin{definition}
		Let $j \colon Z \hookrightarrow X$ be a closed immersion of schemes.
		We say that $(E, \varphi)$ is a \emph{$Z$-trivial line bundle} if $E$ is a line bundle on $X$ and $\varphi \colon E|_{Z} \to \mathcal{O}_Z$ is an isomorphism of $\mathcal{O}_Z$-modules.
		A homomorphism of $Z$-trivial line bundles $u \colon (E, \varphi) \to (F, \psi)$ is a homomorphism of $\mathcal{O}_X$-modules such that $\psi \circ u|_Z=\varphi$.
	\end{definition}

	\begin{proposition}\label{p-obs-triv-lb}
		Let $i \colon (Y_0, Z_0) \to (Y, Z)$ be a thickening of order one given by an ideal $I$ of square zero.	Let $(E, \varphi)$ and $(F, \psi)$ be $Z$-trivial line bundles and \[u_0 \colon (E_0, \varphi_0) \to (F_0, \psi_0)\] a homomorphism of $Z_0$-trivial line bundles. Then there is an obstruction class 
		\[o(u_0, i) \in H^1(Y_0, I \otimes \mathcal{H}om (E_0, F_0 \otimes \mathcal{I}_{Z_0}))\] 
		whose vanishing is necessary and sufficient for the existence of a lifting $u$ of $u_0$. 
		Moreover the set of homomorphisms $u$ lifting $u_0$ is an affine space under $H^0(Y_0,  I \otimes \mathcal{H}om(E_0, F_0 \otimes \mathcal{I}_{Z_0}))$.
		
		Let $(L_0, \varphi_0)$ be a $Z_0$-trivial line bundle on $Y_0$. Then there is an obstruction class 
		\[o(L, \varphi, i) \in H^2(Y_0, I \otimes \mathcal{I}_{Z_0})\] whose vanishing is necessary and sufficient for the existence of a lifting $(L,\varphi)$ of $(L_0, \varphi_0)$ to $(Y,Z)$.
	\end{proposition}
	
	\begin{proof}
		To construct $o(u_0,i)$, we first note that, if $u$ and $v$ are two extension of $u_0$, then $u-v \in H^0(Y_0, I \otimes \mathcal{H}om(E_0, F_0 \otimes \mathcal{I}_{Z_0}))$.
		As extensions of $u_0$ exist locally, we  can construct a torsor $P$ under $I \otimes \mathcal{H}om(E_0, F_0 \otimes \mathcal{I}_{Z_0})$ on $Y_0$ whose sections over an open set $U$ of $ Y_0$ are the $\mathcal{O}_{Y}$-linear extension of $u$ compatible with the trivialisation $\varphi$ and $\psi$.
		Now, as in the proof of \cite[Theorem 8.5.3]{FAG} the class of $P \in H^1(Y_0, I \otimes \mathcal{H}om(E_0, F_0 \otimes \mathcal{I}_{Z_0}))$ is the obstruction class $o(u_0,i)$. 
		To prove the rest of the proposition, we can argue as in \cite[Proof of Theorem 8.5.3]{FAG}.
	\end{proof}

	\begin{corollary} \label{c-lift-lb-triv}
		Let $(A, \mathfrak{m})$ be a Noetherian complete local ring with residue field $k$. Let $j \colon Z\subset X$ be a closed immersion of $k$-schemes
		and let $\mathfrak{Z} \subset \mathfrak{X}$ be a closed immersion of formal schemes over $\Spf(A)$, extending $j$.
		If $H^2(X, \mathcal{I}_Z)=0$, then every $Z$-trivial line bundle $(L, \varphi)$ lifts to a $\mathfrak{Z}$-trivial line bundle $(\mathfrak{L}, \widetilde{\varphi})$ on $(\mathfrak{X},\mathfrak{Z})$.
	\end{corollary}
	
	\begin{proof}
		Using \autoref{p-obs-triv-lb}, we can repeat the same proof as of \cite[Corollaries 8.5.5 and 8.5.6]{FAG}.
	\end{proof}
	
	\section{Lifting snc pairs on globally $F$-split varieties}
	
	In this section, we prove some results on the liftability of smooth globally $F$-split pairs over the ring of Witt vectors valid in all dimensions.
	
\subsection{Lifting over $W_2(k)$}
	
In this subsection, we show \autoref{p-liftpair-W2}. We stress that the pair $(Y,E)$ in the statement of this result is not required to be globally $F$-split.
\autoref{p-liftpair-W2} has been already proven by Achinger-Zdanowicz \cite[Lemma 5.2.2]{AZ21} but we include the following proof for the readers who are not familiar with log structures.
	
\begin{proof}[Proof of \autoref{p-liftpair-W2}]
	Consider the short exact sequence
	\[0\longrightarrow\MO_Y\longrightarrow F_\ast\MO_Y\longrightarrow B^1_Y\longrightarrow 0.\]
	By applying $\textup{Hom}(\Omega^1_Y(\log E),-)$ and taking the induced long exact sequence we get the exact sequence:
	\begin{small}
	   	\[\textup{Ext}^1(\Omega^1_Y(\log E),F_\ast \MO_Y)\rightarrow\textup{Ext}^1(\Omega^1_Y(\log E),B^1_Y)\xrightarrow{\delta} \textup{Ext}^2(\Omega^1_Y(\log E),\MO_Y).\] 
	\end{small}
	By \cite[Variant 3.3.2]{AWZ21}, there is an obstruction class $o_{(Y,E,F)} \in \mathrm{Ext}^1(\Omega_Y^1(\log E), B^1_Y)$ for the lifting of $(Y,E)$ together with the Frobenius morphism $F_Y$. Let $o_{(Y,E)} \in \mathrm{Ext}^2(\Omega^1_Y(\log E), \MO_Y)$ be the obstruction class for the lifting of the pair $(Y,E)$ to $W_2(k)$.
	We show the following compatibility of obstruction classes: 
	\begin{claim}\label{claim-compatibility}
		$\delta(o_{(Y,E,F)})=o_{(Y,E)}$ 
	\end{claim}
	
	\begin{proof}[Proof of \autoref{claim-compatibility}]
		Let $\{U_i\}_{i}$ be an affine open covering of $Y$ and define $U_{ij}\coloneqq U_{i}\cap U_{j}$ and $U_{ijk}\coloneqq U_{i}\cap U_{j}\cap U_{k}$.
		Since $(Y,E)$ is log smooth, there exists a $W_2(k)$-lifting $(\tilde{U}_i, \tilde{E_i})$ of $(U_i, E|_{U_i})$ with the Frobenius morphism $\tilde{F}_i$ for each $i$.
		By \cite[Proposition 8.23]{EV92}, there exists an isomorphism $\phi_{ij}\colon(\tilde{U}_i, \tilde{E_i})|_{U_{ij}}\cong (\tilde{U}_j, \tilde{E_j})|_{U_{ij}}$ over $(U_{ij}, E|_{U_{ij}})$.
		Then $\phi_{ijk}=\phi_{ki}\circ\phi_{jk}\circ\phi_{ij}$ is an infinitesimal automorphism of $(\tilde{U}_i, \tilde{E_i})|_{U_{ijk}}$, and hence we can take a corresponding derivation \[\psi_{ijk}\in \Hom(\Omega^1_{U_{ijk}}(\log\,E), \MO_{U_{ijk}})\] by \cite[Proposition 8.22]{EV92}.
		Note that we have the equation $\phi_{ijk}=\mathrm{id}+p\psi_{ijk}$.
		We can see that  $o_{(Y,E)}=\{\psi_{ijk}\}_{ijk}\in \mathrm{Ext}^2(\Omega^1_Y(\log\,E),\MO_Y)$ (cf.~\cite[Theorem 8.5.9]{FAG} and \cite[Theorem 2.3]{KN20}). 
			
		Since $\phi_{ij}^{-1}\tilde{F}_j\phi_{ij}$ and $\tilde{F}_i$ are both $W_2(k)$-liftings of the Frobenius morphism of $U_{ij}$, there exists $\eta_{ij}\in \Hom(\Omega^1_{U_{ij}}(\log\,E), F_{*}\MO_{U_{ij}})$ such that $\phi_{ij}^{-1}\tilde{F}_j\phi_{ij}-\tilde{F}_i=p\eta_{ij}$ by \cite[Proposition 9.9]{EV92}.  
		
		We define $\overline{\eta}_{ij}\in \mathrm{Hom}(\Omega^1_{U_{ij}}(\log\,E), B_{U_{ij}})$ to be the natural image of $\eta_{ij}$ under the morphism $\mathrm{Hom}(\Omega^1_{U_{ij}}(\log\,E), F_*\mathcal{O}_{U_ij}) \to \mathrm{Hom}(\Omega^1_{U_{ij}}(\log\,E), B_{U_{ij}})$. Then we can see by \cite[Variant 3.3.2]{AWZ21} that  \[o_{(Y,E,F)}=\{\overline{\eta}_{ij}\}_{ij} \in \mathrm{Ext}^1(\Omega^1_{U_{ij}}(\log\,E), B_{U_{ij}}).\]
		In particular, $\overline{\eta}_{ij}+\overline{\eta}_{jk}+\overline{\eta}_{ki}=0$ on $U_{ijk}$ 
		and so there exists $\eta_{ijk}\in \mathrm{Hom}(\Omega^1_{U_{ijk}}(\log\,E), \MO_{U_{ijk}})$ such that $\eta_{ijk}^{p}=\eta_{ij}+\eta_{jk}+\eta_{ki}$. 
		By construction of the boundary map, \[\delta(o_{(Y,E,F)})=\{\eta_{ijk}\}_{ijk}\in \mathrm{Ext}^2(\Omega^1_Y(\log\,E),\MO_Y).\]
			
		Since $\phi_{ij}^{-1}\tilde{F}_j\phi_{ij}-\tilde{F}_i=p\eta_{ij}$, it follows that $\phi_{ijk}^{-1}\tilde{F}_i\phi_{ijk}-\tilde{F}_i=p(\eta_{ij}+\eta_{jk}+\eta_{ki})$. 
		On the other hand, as in \cite[Appendix, Proposition 1 (iv)]{MS87}, we get that $\phi_{ijk}^{-1}\tilde{F}_i\phi_{ijk}-\tilde{F}_i=p\psi_{ijk}^{p}$.  
		Specifically for a local section $y$ of $ \cO_{\tilde{U}_{ijk}}$,
		\begin{align*}
        \phi_{ijk}^{-1}\tilde{F}_i\phi_{ijk}(y)&= (\mathrm{id} - p\psi_{ijk})\tilde{F}_i(\mathrm{id} + p\psi_{ijk})(y) 	 \\
    &= (\mathrm{id} - p\psi_{ijk})(\tilde{F}_i(y) + p \psi_{ijk}(y)^p)\\
    &= \tilde{F}_i(y) - p\psi_{ijk}(y^p)  + p\psi_{ijk}(y)^p  \\
    &= \tilde{F}_i(y) + p\psi_{ijk}(y)^p,
			\end{align*}
	where the last equality follows from $\psi_{ijk}(y^p)=0$ as $\psi_{ijk}$ is a derivation. We also repeatedly used that $p^2=0$.

    We can now conclude that $o_{(Y,E)}=\{\psi_{ijk}\}_{ijk}=\{\eta_{ijk}\}_{ijk}=\delta(o_{(Y,E,F)})$.
    \end{proof}

Since $Y$ is globally $F$-split, $\delta$ is the zero homomorphism.
Therefore the obstruction class $o_{(Y,E)}$ vanishes concluding the proof.
\end{proof}

\subsection{Lifting Fano varieties}

	In what follows,  we show an application of \autoref{p-liftpair-W2} to the lifting of snc pairs over $W(k)$ whose underlying variety is a smooth globally $F$-split Fano(-type) variety.
	First, we recall the Kodaira-Akizuki-Nakano vanishing theorem for snc pairs admitting a lifting to $W_2(k)$ proven in \cite{Ha98}. 
	
	\begin{theorem} \label{p-hara}
		Let $(Y,E)$ be an snc pair of dimension $d$ which admits a lifting over $W_2(k)$.
		Let $A$ be an ample $\mathbb{Q}$-divisor whose fractional part $(A-\lfloor A \rfloor)$ of $A$ is contained in $E$.
		If $p \geq d$, then
		\begin{enumerate}
			\item[(a)] $H^j(Y, \Omega_Y^{i}(\log E) \otimes \mathcal{O}_Y(-E-\lfloor{-A \rfloor}))=0$ if $i+j>d$;
			\item[(b)] $H^j(Y, \Omega_Y^{i}(\log E) \otimes \mathcal{O}_Y(-\lceil{A \rceil}))=0$ if $i+j<d$.
		\end{enumerate}
	\end{theorem}
	
	\begin{proof}
		Assertion (a) is \cite[Corollary 3.8]{Ha98}. The case $p=d$ holds because the proof of  \cite[Corollary 3.8]{Ha98} uses the hypothesis $p>d$ only for the quasi-isomorphism \[\bigoplus \Omega_Y^i(\log E)[-i] \cong F_*\Omega^{\bullet}_Y(\log E),\] which is true also for $p=d$ by \cite[Proposition 10.19]{EV92}.
		
		As for (b), recall that the natural pairing $\Omega_Y^i(\log E) \otimes \Omega_Y^{d-i}(\log E) \to \omega_Y(E)$ is non-degenerate and therefore $\Omega_Y^i(\log E) \cong (\Omega_Y^{d-i}(\log E))^{\vee} \otimes \omega_Y(E)$.
		By Serre duality the following isomorphisms hold: 
		\begin{small}
			\begin{align*}
				H^j(Y, \Omega_Y^{i}(\log E) \otimes \mathcal{O}_Y(-\lceil{A \rceil})) & \cong H^j(Y,(\Omega_Y^{d-i}(\log E))^{\vee} \otimes  \omega_Y(E -\lceil{A \rceil})) \\
				& \cong  H^{d-j}(Y, \Omega_Y^{d-i}(\log E)\otimes \MO_Y(-E+\lceil{A \rceil}))^{\vee}.
			\end{align*}
		\end{small}
		Since $\lceil{A \rceil}=-\lfloor{-A\rfloor}$ we conclude by (a).
	\end{proof}
	
	\begin{proposition} \label{p-lift-Fanopairs}
		Let $Y$ be a smooth globally $F$-split projective variety over $k$ of dimension $d$. 
		Suppose there exists an effective $\Q$-divisor $\Delta$ such that
		\begin{enumerate}
			\item[\textup{(1)}] $\lfloor \Delta \rfloor=0$ and  $(Y, \Supp(\Delta))$ is snc;
			\item[\textup{(2)}] $-(K_Y+\Delta)$ is ample.		
       \end{enumerate}
		Let $E$ be an snc reduced divisor containing $\Supp(\Delta)$.
		If $p \geq d$, then 
		\begin{enumerate}
			\item[(a)] $H^i(Y, \mathcal{O}_Y)=0$ for $i>0$;
			\item[(b)] $H^2(Y, T_Y(-\log E))=0$.
		\end{enumerate}
		In particular, $(Y,E)$ lifts over $W(k)$.
	\end{proposition}
	
	\begin{proof}
		By \autoref{p-liftpair-W2}, the pair $(Y,E)$ lifts over $W_2(k)$ so we can apply \autoref{p-hara}.
		Let us choose the ample $\mathbb{Q}$-divisor $A\coloneqq -K_Y-\Delta$. Note that $\lfloor{ -A \rfloor}=K_Y$ and $\lceil{ A \rceil}=-K_Y$.
		To show (a), it is sufficient to notice that
		\[H^i(Y, \mathcal{O}_Y)=H^i(Y, \omega_Y(-\lfloor{ -A \rfloor}))=H^i(Y, \omega_Y(E)(-E-\lfloor{ -A \rfloor})),\]
		vanishes for $i>0$ by \autoref{p-hara}.a.

		We prove (b). As 
        \[
        H^{d-2}(Y,\Omega^1_Y(\log E) \otimes \omega_Y)\cong H^{d-2}(Y,\Omega^1_Y(\log E) \otimes \mathcal{O}_Y(-\lceil A \rceil)) 
        \]
		vanishes by \autoref{p-hara}.b, we deduce $ H^2(Y, T_Y(-\log E))=0$  by Serre duality.
		
		For the last assertion, note that $H^2(Y, T_Y(-\log E))$ is the obstruction space to the existence of a formal log lifting of $(Y,E)$ over $W(k)$ by \cite[Proposition 8.6]{log-deformation(Kato)} (cf.~\cite[Theorem 2.3]{KN20}).
		Moreover, any formal lifting of $(Y,E)$ is algebraisable as $H^2(Y, \mathcal{O}_Y)=0$ by (a) and \cite[Corollary 8.5.6 and Corollary 8.4.5]{FAG}.
	\end{proof}
	
	\section{Log liftability of globally $F$-split surface pairs}
	
	In this section we prove the log liftability of globally $F$-split surface pairs (\autoref{t-log-lift-GFS}). 
	We divide the proof in two cases. 
	In \autoref{s-K0-klt} we show log liftability of klt Calabi--Yau surfaces.
	We discuss the remaining cases (where $(X,D)$ is not klt or $K_X+D$ is not pseudo-effective) in \autoref{s-not-pseff}. 
	
	Throughout this section, $k$ denotes an algebraically closed field of characteristic $p>0$.
	
	\subsection{$K$-trivial surfaces with klt singularities}\label{s-K0-klt}
	
	We start by proving log liftability over $W(k)$ of globally $F$-split Calabi--Yau surfaces with canonical singularities. For this, we rely on the Enriques-Kodaira classification of their minimal resolutions (\cite{BM77}) and special properties of canonical liftings of their minimal models (\cite{LT19},\cite{MS87},\cite{Nyg83}).
	Then we are able to  conclude the general klt Calabi--Yau case by, roughly speaking, considering canonical covers.
	
	Recall that $\mathbb{Q}$-factorial proper surfaces are projective by \cite[Corollary 4, page 328]{Kle66}.

	\subsubsection{Ordinary K3 surfaces}\label{ss-K3}

	In what follows, a smooth proper surface $Y$ over $k$ is called a K3 surface if $K_Y \sim 0$ and $h^1(Y, \mathcal{O}_Y)=0$. 
	A K3 surface $Y$ called \emph{ordinary} if the induced action of the Frobenius on its top cohomology $F \colon H^2(Y, \MO_Y) \to H^2(Y, \MO_Y)$ is bijective.
	The following shows that ordinarity  coincides with $Y$ being globally $F$-split.
	
	\begin{lemma}\label{l-ordinary}
		Let $Y$ be a normal Gorenstein proper variety over $k$ of dimension $n$ such that $K_Y \sim 0$. 
		Then the following are equivalent:
		\begin{enumerate}
			\item[(a)]  $F \colon H^n(Y, \MO_Y) \to H^n(Y, F_*\MO_Y)$ is bijective;
			\item[(b)] $\Tr \colon H^0(Y, F_*\omega_Y) \to H^0(Y, \omega_Y)$ is bijective, where $\Tr$ is the Frobenius trace map;
			\item[(c)] $Y$ is globally $F$-split.
		\end{enumerate} 
	\end{lemma}
	\begin{proof}
	See \cite[Proposition 2.6]{PZ21} (cf.~also \cite[Proposition 9]{MR85}).
	\end{proof}
	
	Given an ordinary K3 surface $Y$, in \cite{Nyg83} Nygaard shows the existence of a \emph{canonical lifting} $\mathcal{Y}_{\can}$ of $Y$ over $W(k)$.
	We recall some of its properties that we will use:
	
	\begin{proposition} \label{p-lift-Pic-K3}
		Let $Y$ be a globally $F$-split K3 surface over $k$ and let $\mathcal{Y}_\can$ be its canonical lifting constructed in \cite{Nyg83}. Then 
		\begin{enumerate}
			\item[(1)] every automorphism $\varphi$ of $Y$ lifts uniquely to an automorphism $\widetilde{\varphi} \colon \mathcal{Y}_{\can} \to \mathcal{Y}_{\can}$ over $W(k)$;
			\item[(2)] $\Pic(\mathcal{Y}_\can) \to \Pic(Y)$ is an isomorphism of abelian groups.
		\end{enumerate}
		In particular, $\mathcal{Y}_{\can}$ is projective over $W(k)$.
	\end{proposition} 
	\begin{proof}
		The existence part of (1) is proven in \cite{Sri19} and \cite[Proposition 2.3]{LT19}. 
		The uniqueness follows from the vanishing of the tangent space $T_{\text{id}} \Aut_{Y/k} \cong H^0(Y, T_Y)=0$ of the automorphism scheme at the identity  (see \cite[Theorem 7]{RS76}, \cite{Nyg79} and \cite[Corollary 1.1]{Martin22}).
		For (2), we refer to the proof of \cite[Proposition 1.8]{Nyg83}.
		The last assertion follows from \cite[Theorem 8.4.10]{FAG}.
	\end{proof}

	\begin{proposition}\label{prop:K3_lift_div}
		Let $Y$ be a globally $F$-split K3 surface and suppose $(Y,D)$ is an snc pair. Then there exists a subscheme $\mathcal{D}$ of the canonical lifting $\mathcal{Y}_{\can}$ such that $(\mathcal{Y}_{\can},\mathcal{D})$ is a lifting of $(Y,D)$ over $W(k)$.
		
		In particular, if $X$ is a globally $F$-split surface such that the minimal resolution $f \colon Y \to X$ is a K3 surface, then $(Y, \Ex(f))$ admits a canonical lifting $(\mathcal{Y}_{\can}, \mathcal{E}_{\can})$ over $W(k)$.
	\end{proposition}
		Note that if a surface $X$ has the minimal resolution $f\colon Y\to X$ such that $Y$ is a K3 surface, then $X$ has canonical singularities as follows:
        We have $K_Y=f^{*}K_X-E$ for some effective $f$-exceptional divisor $E \geq 0$. Since $-K_X$ is $\mathbb{Q}$-effective (\autoref{l-Q-K+delta} (a)), we obtain $E=0$. Thus, $X$ has canonical singularities. 
	\begin{proof}
		Let $D_1, \dots, D_n$ be the irreducible components of $D$.
		By \autoref{p-lift-Pic-K3},   $\MO_Y(D_i)$ lifts to a line bundle $\mathcal{L}_i$ on the canonical lifting $\mathcal{Y}_{\can}$ for every $i=1, \dots, n$. 
		We show it is sufficient to prove, similarly to \cite[Lemma 2.3]{LM18}, that the natural restriction map. 
		\[H^0(\mathcal{Y}_{\can},\mathcal{L}_i)\longrightarrow H^0(Y,\MO_Y(D_i))\]
		is surjective for every $i$. 

		Indeed, if surjectivity holds, then there exists an effective Cartier divisor $\mathcal{D}_i$ such that $\mathcal{D}_i|_Y=D_i$. 
		By \autoref{l-flat-Cartier}, $\mathcal{D}_i$ is flat over $W(k)$ and we thus conclude by \autoref{l-loc-triv}.
		
		To show surjectivity of the restriction map it is enough to show $H^i(\mathcal{Y}_{\can},\mathcal{L}_i)=0$ for all $i>0$ and apply cohomology and base change \cite[Theorem III.12.11]{Ha77}. 
		By upper semi-continuity \cite[Theorem III.12.8]{Ha77}, it is enough to show $H^i(Y,\MO_Y(D_i))=0$ for $i>0$. By Serre duality $H^2(Y, \MO_Y(D_i))=H^0(Y, \MO_Y(-D_ i))^{\vee}=0$.
		Finally, $H^1(Y, \MO_Y(D_i))=0$: indeed $\mathcal{O}_{D_i}(D_i) \cong \omega_D$ by adjunction, then we take the exact sequence \[
		0=H^1(Y, \MO_Y) \to H^1(Y,\MO_Y(D_i)) \to H^1(D_i, \omega_{D_i}) \to H^2(Y, \MO_Y) \to 0
		\]
		and since the last two terms are one-dimensional we conclude that $H^1(Y,\MO_Y(D_i))=0$.

        To prove the last assertion, as $Y$ is globally $F$-split by \autoref{l-GFR-pullback} and $(Y,\Ex(f))$ is snc, there exists a lifting $\mathcal{E}_i$ for every irreducible component $E_i \subset \Ex(f)$. Note that the lifting $\mathcal{E}_i$ is unique as $H^0(Y, \mathcal{O}_Y(E_i))$ is one-dimensional.  We define $\mathcal{E}_{\can}:=\sum_i \mathcal{E}_i$.
	\end{proof}
	
	\begin{remark}
		Note that \autoref{prop:K3_lift_div} fails for certain supersingular K3 surfaces in characteristic $p \leq 19$ constructed in \cite[Theorem 1]{Shi04} as explained in \cite[Remark 3.4]{Kaw21}.
	\end{remark}
	
	\subsubsection{Globally $F$-split Enriques surfaces} \label{ss-enriques}
	We briefly recall the classification of Enriques surfaces in characteristic $p>0$ and we refer the reader to \cite{BM76, LT19} for a more detailed treatment. 
	In what follows, a smooth projective surface $X$ over $k$ is called \emph{Enriques} if $K_X \equiv 0$ and the $2^{\textrm{nd}}$ \'etale Betti number $b_2(X)=10$. In particular, one can check that $\chi(X, \cO_X)=1$ (see \cite[\S 3]{BM76}).
	
	We say that an Enriques surface $X$ is:
	\begin{enumerate}
	    \item[(a)] \emph{classical} if $h^1(X,\cO_X)=0$ (in this case, $K_X \not \sim 0$ and $2K_X \sim 0$);
	    \item[(b)] \emph{singular}  if $h^1(X,\cO_X)=1$ (hence, $K_X \sim 0$) and the Frobenius morphism acts bijectively on $H^1(X,\cO_X)$;
        \item[(c)]\emph{supersingular}:
	     if $h^1(X,\cO_X)=1$ (hence, $K_X \sim 0$) and the Frobenius morphism acts trivially on $H^1(X,\cO_X)$. 
	    \end{enumerate}
    If $p > 2$, then all Enriques surfaces are classical.
	However, if $p=2$, then classical, singular,
	and supersingular ones form three disjoint non-empty classes. 
	Moreover, by \cite[Theorem 2, p.216]{BM76} every Enriques surface $X$ admits a canonical $G$-torsor $\pi \colon Z \to X$ with
	\begin{alignat*}{2}
	    G &= \mu_2\qquad &&\textrm{ when } X \textrm{ is classical,} \\
	    G &= \mathbb{Z}/2\mathbb{Z} && \textrm{ when } X \textrm{ is singular,} \\
	    G &= \alpha_2 &&\textrm{ when } X \textrm{ is supersingular}.
	\end{alignat*}
	We will call $\pi \colon Z \to X$ the \emph{canonical double covering} of $X$. 
	Note that $\mu_2$ is isomorphic to $\mathbb{Z}/2\mathbb{Z}$ as group schemes when $p>2$.
    We now relate these notions to global $F$-splitting.
	
	\begin{lemma}\label{l:doublecoverEnriques}
    The following hold.
	\begin{enumerate}
	    \item[(a)] Suppose $p=2$. Then an Enriques surface $X$ over $k$ is globally $F$-split if and only if it is singular.
	    \item[(b)]  In general, an Enriques surface $X$ over $k$ is globally $F$-split if and only if the canonical double covering $\pi \colon Z \to X$ is \'etale and $Z$ is an ordinary K3 surface.
	\end{enumerate}
\end{lemma}

\begin{proof}
    We start with (a). First, suppose that $X$ is globally $F$-split. Then 
    \[
	H^0(X, \cO_X((1-p)K_X)) \neq 0,
	\]
	and therefore $K_X \sim 0$ as $p=2$. 
	Moreover, the existence of an $F$-splitting implies that $F \colon H^1(X,\cO_X) \to H^1(X,\cO_X)$ splits, and so it is a bijection. In particular, $X$ is singular.
	
	 As for the opposite implication, suppose that $X$ is singular. 
	 Let $\pi \colon Z \to X$ be the canonical double covering of $X$, which is \'etale. By \cite[Theorem 2.7]{Crew84}, $Z$ is an ordinary K3 surface, which in this article means that the dimension of the slope-$0$ crystalline cohomology 
	\[
	\dim_K H^2_{\mathrm{cris}}(Z/W)_0=1.
	\]
	By \cite[7.2(a) p.653]{Illusie79}, this is equivalent to $h :=\dim_K (H^2(Z, W\cO_Z) \otimes K)$ being equal to $1$. 
	In turn, by \cite[Theorem 4.5]{Yobuko19} (or \cite[Lemma 1.3]{Nyg83} and \autoref{l-ordinary}), this is equivalent to $Z$ being globally $F$-split.
    As $\pi$ is \'etale and $\omega_X \cong \mathcal{O}_X$, we conclude that $H^0(Z, \omega_Z) \cong H^0(X, \omega_X)$ and thus the action of the Frobenius on $H^0(X, \omega_X)$ is bijective and thus $X$ is globally $F$-split by \autoref{l-ordinary}.

	We now prove (b).
	Suppose first that $X$ is globally $F$-split. By (a), the canonical cover is \'etale in all characteristic and thus we conclude by \autoref{l-F-split-quasietale} and \autoref{l-ordinary}.
	
	Suppose now that $\pi \colon Z \to X$ is \'etale and $Z$ is ordinary. By \autoref{l-ordinary}, $Z$ is globally $F$-split.
    If $p=2$, we proved in (a) that $X$ is ordinary.
	If $p>2$, fix a splitting $\psi \colon \mathcal{O}_Z \to F_*\mathcal{O}_Z$ and consider the following commutative diagram: 
\[ \label{eq: diagram_enriques}
		\xymatrix{
			\mathcal{O}_X \ar[r] \ar[d]  & F_*\mathcal{O}_X \ar[d] \\
			\pi_*\mathcal{O}_Z \ar[r] \ar@/^1pc/[u]^{\frac{\Tr}{2}} & \pi_*F_*\mathcal{O}_Z \ar@/_1pc/[l]_{\pi_* \psi},
	}
\]
	where $\frac{\Tr}{2} \colon \pi_* \mathcal{O}_Z \to \mathcal{O}_X$ is a splitting. By following the diagram, we conclude $X$ is globally $F$-split.
\end{proof}
	
	In particular, an Enriques surface $X$ is globally $F$-split if and only if it is ordinary in the sense of \cite[Definition 2.1]{LT19}.
	
	The following states the conditions for a line bundle to descend under a Galois \'etale morphism.

	\begin{lemma}\label{l-descent-equivariant}
		Let $f \colon X \to Y$ be a Galois finite \'etale morphism of integral schemes and let $G$ be its Galois group.
		Let $L$ be a $G$-equivariant line bundle on $X$. 
		Then there exists a unique line bundle $M$ on $Y$ such that $f^*M$ is isomorphic to $L$ as $G$-equivariant line bundles.
	\end{lemma}
	
	\begin{proof}
		See \cite[\href{https://stacks.math.columbia.edu/tag/023T}{Tag 023T}]{stacks-project} and \cite[\href{https://stacks.math.columbia.edu/tag/023T}{Tag 05B2}]{stacks-project}) (cf.\ \cite[Th\'eor\`eme 2.3]{DN89}).
	\end{proof}
	
	We recall the notion of a canonical lifting for globally $F$-split Enriques surfaces introduced in \cite[Definition 2.5]{LT19}.
	
	\begin{proposition}\label{p-canonical-lift-enriques}
		Let $Y$ be a globally $F$-split Enriques surface and let $\pi \colon Z \to Y$ be the canonical double covering. Then there exists a projective lifting $\mathcal{Y}_\can$ of $Y$ over $W(k)$ together with a lifting $\widetilde{\pi}\colon \mathcal{Z}_{\can} \to \mathcal{Y}_{\can}$ of $\pi$ over $W(k)$ such that:
		\begin{enumerate}
			\item[(1)] $\mathcal{Z}_{\can}$ is the canonical lifting of $Z$;
			\item[(2)] $\widetilde{\pi}$ is a Galois finite \'etale cover of degree $2$;
			\item[(3)] $\Pic(\mathcal{Y}_{\can}) \to \Pic(Y)$ is an isomorphism of abelian groups.
		\end{enumerate}
		We say that $\mathcal{Y}_{\can}$ is the canonical lifting of the Enriques surface $Y$.
	\end{proposition}
	
	\begin{proof}
	By \autoref{l:doublecoverEnriques}, $\pi$ is \'etale and $Z$ is an ordinary K3 surface. Thus (1) and (2) are proven in \cite[Theorem 2.4]{LT19}. 
		For the proof of (3), let $L$ be a line bundle on $Y$. Note that $M:=\pi^*L$ extends to a unique line bundle $\mathcal{M}$ on $\mathcal{Z}_{\can}$ by \autoref{p-lift-Pic-K3}.
		The group of $W(k)$-automorphisms of $\widetilde{\pi}$ is $\mathbb{Z}/2\mathbb{Z}$.
		We claim that $\mathcal{M}$ is $(\mathbb{Z}/2\mathbb{Z})$-equivariant. 
		Clearly $M$ is $(\mathbb{Z}/2\mathbb{Z})$-equivariant line bundle on $Z$. 
	    A lifting $\mathcal{M}$ of $M$ to $\mathcal{Z}_{\can}$ is unique by \autoref{p-lift-Pic-K3} and thus it must be $(\mathbb{Z}/2\mathbb{Z})$-equivariant.
		Therefore $\mathcal{M}$ descends to a line bundle $\mathcal{L}$ on $\mathcal{Y}_{\can}$ by \autoref{l-descent-equivariant}.
		Since $\pi^{*}L \cong \widetilde{\pi}^{*}\mathcal{L}|_{Y}$, it follows from \autoref{l-descent-equivariant} that $L \cong \mathcal{L}|_Y$, so $\mathcal{L}$ is a lifting of $L$. 
	\end{proof}
	
	\begin{proposition}\label{prop:Enriques_lift_div}
		Let $X$ be a projective globally $F$-split surface over $k$ with canonical singularities.
		Let $f\colon (Y,E) \to X$ be the minimal resolution. 
		Suppose that $Y$ is an Enriques surface.
		Then $(Y,E)$ admits a lifting $(\mathcal{Y}_\can, \mathcal{E}_\can)$ over $W(k)$ where $\mathcal{Y}_{\can}$ is the canonical lifting of $Y$.
	\end{proposition}
\begin{proof}
	By \autoref{l-GFR-pullback}, $Y$ is globally $F$-split. 
	Thus, by \autoref{l:doublecoverEnriques}, there exists an \'etale double cover $\pi \colon Z\longrightarrow Y$ where $Z$ is a globally $F$-split K3 surface. Let $\widetilde{\pi} \colon  \mathcal{Z}_{\can} \to \mathcal{Y}_{\can}$ be the lifting over $W(k)$ given by \autoref{p-canonical-lift-enriques} and denote by $i$ the natural involution on $\mathcal{Z}_\can$.
	We claim that each irreducible component $D$ of $E$ lifts to a subscheme $\mathcal{D} \subset \mathcal{Y}_\can$.
		
		Since $D \cong \mathbb{P}^1$ is simply connected, the preimage $ \pi^{-1}D$ will consist of two disjoint divisors $F  \sqcup G$. Let $L:=\mathcal{O}_Y(D)$ and $L_Z:=\pi^*L=\MO_Z(F +G).$
		Let $\mathcal{L}_{\mathcal{Z}_{\can}}$ be the canonical lifting of $L_{Z}$ to $\mathcal{Z}_{\can}$ guaranteed by \autoref{p-lift-Pic-K3}.
		By \autoref{prop:K3_lift_div}, there exist unique liftings $\mathcal{F}_\can$ and $\mathcal{G}_\can$ of $F$ and $G$ inside $\mathcal{Z}_\can$. 
		If $f \in H^0(\mathcal{Z}_\can, \mathcal{O}_{\mathcal{Z}_\can}(\mathcal{F}_\can))$ defines $\mathcal{F}_\can$,
		then $g:=i^*(f)$ belongs to $H^0(\mathcal{Z}_\can, \mathcal{O}_{\mathcal{Z}_\can}(\mathcal{G}_\can))$ by uniqueness of lifts of line bundles as $g|_{Z} \in H^0(Z, \mathcal{O}_Z(G))$.
		Then $s=f \cdot  i^*f \in H^0(\mathcal{Z}_\can, \mathcal{L}_{\mathcal{Z}_{\can}})$ is a section defining the divisor $\mathcal{F}_\can+\mathcal{G}_\can$.
	    As $s$ is $(\mathbb{Z}/2\mathbb{Z})$-invariant (indeed, $i^2=\mathrm{id}$), it descends to a section $t \in H^0(\mathcal{Y}_\can, \mathcal{L}_{\mathcal{Y}_\can})$ by \ 
		\cite[\href{https://stacks.math.columbia.edu/tag/03DW}{Tag 03DW}]{stacks-project}, where $\mathcal{L}_{\mathcal{Y}_\can}$ is the lifting of $L$ constructed in \autoref{p-canonical-lift-enriques}. 
		The Cartier divisor $\mathcal{D} \subset \mathcal{Y}_{\can}$ cut out by $s$ gives then the desired lifting of $D$ by \autoref{l-flat-Cartier}.
\end{proof}

\subsubsection{General case}

We recall the properties of the canonical lifting of a globally $F$-split abelian variety.
	
	\begin{theorem}[cf. {\cite[Theorem 1, Appendix]{MS87}}]\label{t-ms}
		Let $A$ be a globally $F$-split abelian variety.
		Then there exists a canonical lifting $\mathcal{A}_{\can}$ of $A$ over $W(k)$ such that 
		\begin{enumerate}
			\item[(a)] the Frobenius morphism $F$ lifts to a morphism $F_{\mathcal{A}_{\can}} \colon \mathcal{A}_{\can} \to \mathcal{A}_{\can} $  and the lifting $(\mathcal{A}_{\can}, F_{\mathcal{A}_{\can}})$ is unique up to unique isomorphism; 
			\item[(b)] for every $f \in \Aut(A)$, there exists a unique automorphism $f_{\can} \in \Aut(\mathcal{A}_{\can})$ lifting $f$ over $W(k)$ such that $f_{\can} \circ F_{\mathcal{A}_{\can}} = F_{\mathcal{A}_{\can}} \circ f_{\can}$;
			\item[(c)] the natural restriction morphism \[\Pic(\mathcal{A}_{\can})_{F_{\mathcal{A}_{\can}}}:=\left\{ \mathcal{L} \in \Pic(\mathcal{A}_{\can}) \mid F^*\mathcal{L} \cong \mathcal{L}^{\otimes p} \right\} \to \Pic(A)\] 
			is an isomorphism.
		\end{enumerate}
	In particular, $\mathcal{A}_{\can}$ is projective over $W(k)$.
	\end{theorem}
	
	\begin{definition}
		We say that a smooth projective $k$-variety $X$ is \emph{$Q$-abelian} if there exists an \'etale $k$-morphism $A \to X$ where $A$ is an abelian variety. 
	\end{definition}
	
\begin{remark}
Recall that a finite \'etale cover of an abelian variety is abelian. 
Thus, by \cite[\href{https://stacks.math.columbia.edu/tag/0BN2}{Tag 0BN2}]{stacks-project} and  \cite[\href{https://stacks.math.columbia.edu/tag/0BNB}{Tag 0BNB}]{stacks-project}, we can assume that every $Q$-abelian variety admits an \'etale \emph{Galois} cover $A \to X$ where $A$ is an abelian variety.
\end{remark}

\begin{proposition}\label{lift-quot-abelian}
	Let $X$ be a globally $F$-split smooth projective $Q$-abelian variety and let $\pi \colon A \to X$ be a Galois \'{e}tale morphism with Galois group $G$, where $A$ is an abelian variety. 
	Then
	\begin{enumerate}
		\item[(a)] there exists a canonical lifting $\mathcal{G}_\can \subset \Aut(\mathcal{X}_\can)$ of $G$;
		\item[(b)] the quotient $\widetilde{\pi}\colon \mathcal{A}_{\can} \to \mathcal{X}_{\can}:=\mathcal{A}_{\can}/\mathcal{G}_{\can}$ is an \'etale morphism and it is a lifting of $\pi$;
		\item[(c)] the lifting $\mathcal{X}_{\can}$ does not depend on the choice of the \'etale morphism $\pi$;
		\item[(d)] $\Pic(\mathcal{X}_\can) \to \Pic(X)$ is surjective.
	\end{enumerate} 
	We say that $\mathcal{X}_\can$ is the canonical lifting of $X$.
\end{proposition}
	
\begin{proof}
	By \autoref{l-F-split-quasietale}, $A$ is globally $F$-split and we let $\mathcal{A}_{\can}$ be the canonical lifting over $W(k)$. 
	By \autoref{t-ms} there exists a canonical lifting of $G$ to a group of automorphisms $\mathcal{G}_{\can}$ of $\mathcal{A}_\can$ over $W(k)$, proving (a).
	For (b), we choose the lifting of $\pi$ to be the quotient $\tilde{\pi}\colon \mathcal{A}_{\can} \to \factor{\mathcal{A}_{\can}}{\mathcal{G}_{\can}}$, whose existence is guaranteed by \cite[Expos\'{e} V, Proposition 1.8]{SGA1}.
By construction it is easy to see that $\mathcal{X}_\can$ does not depend on the Galois cover $A \to X$, proving (c).
		
	We are left to prove (d). 
	Let $L$ be a line bundle on $X$ and let $M:=\pi^*L$ be the pull-back on $A$.
	By \autoref{t-ms}, we consider $\mathcal{M}$ to be the unique lifting of $M$ to $\mathcal{A}_{\can}$ belonging to $\Pic(\mathcal{A}_{\can})_{F_{\mathcal{A}_\can}}$.
	By uniqueness of the lifting in $\Pic(\mathcal{A}_{\can})_{F_{\mathcal{A}_\can}}$ and the fact that canonical lifts of automorphisms commute with the lift of Frobenius (\autoref{t-ms}(b)), $\mathcal{M}$ must be $\mathcal{G}$-equivariant and therefore we conclude that $L$ lifts to a line bundle $\mathcal{L}$ on $\mathcal{X}_{\can}$ by \autoref{l-descent-equivariant}.
\end{proof}
	
Finally we prove log liftability of numerically $K$-trivial surfaces with canonical singularities over $W(k)$.
	
\begin{theorem}\label{prop:lift_minres_K0}
	Let $X$ be a globally $F$-split projective surface with canonical singularities.
	Suppose that $K_X \equiv 0$ and let $f \colon (Y, E) \to X$ be the minimal resolution with exceptional divisor $E$. Then
	\begin{enumerate} 
	    \item[(a)] $Y$ is globally $F$-split and either it is 
	        \begin{enumerate}
	            \item[(i)] a K3 surface,
	            \item[(ii)] an Enriques surface, 
	            \item[(iii)] a $Q$-abelian surface;
	        \end{enumerate}
 	\item[(b)] there exists a lifting $\widetilde{f} \colon (\mathcal{Y}_\can, \mathcal{E}_\can) \to \mathcal{X}_\can$ of $f$, where $\mathcal{Y}_\can$ is the canonical lifting of $Y$ defined in \autoref{prop:K3_lift_div}, \autoref{t-ms} and \autoref{p-canonical-lift-enriques};
		\item[(c)] $\Pic(\mathcal{Y}_{\can}) \to \Pic(Y)$ is a surjective homomorphism of abelian groups.
	\end{enumerate} 
\end{theorem}

\begin{proof}
	Using the Enriques classification of smooth projective surfaces over algebraically closed fields of positive characteristic (see \cite{BM76, BM77}), we have to deal with four different cases depending on the Betti numbers: $Y$ is a K3 surface, an Enriques surface, an abelian variety, or a (quasi-)hyperelliptic surface.
	Let us note that $Y$ cannot be quasi-hyperelliptic because, as $Y$ is globally $F$-split, the Albanese morphism $a \colon Y \to E$ is an $F$-split morphism and the general fibre is normal by \cite[Theorem 1.2 and 1.3(4)]{Eji19}. 
	By the classification of hyperelliptic surfaces (see \cite[Theorem 4 and see table at page 37]{BM77}), we see that, except \cite[Case (a3), page 37]{BM77}), $Y$ admits an \'etale cover by an product of elliptic curves.
    In this last case:  $Y \cong (E_1 \times E_2)/(\mathbb{Z}/2\mathbb{Z}) \times \mu_2$ and it is easy to see that $(E_1\times E_0)/\mu_2$ is an abelian variety (specifically, the action of $\mu_2$ commutes with the group structure and does not have a fixed point), and thus $Y$ is $Q$-abelian.
	
	 In what follows we prove (b) and (c): that $(Y,E)$ lifts to $(\mathcal{Y}_{\can}, \mathcal{E}_{\can})$ and that \[\Pic(\mathcal{Y}_{\can}) \to \Pic(Y)\] is surjective. 
		
	Case (i) (K3 surface) follows from \autoref{prop:K3_lift_div} and \autoref{p-lift-Pic-K3}(2).  Case (ii) (Enriques surface) follows from \autoref{prop:Enriques_lift_div} and \autoref{p-canonical-lift-enriques}(3).

	 Before proceeding further, we note that in Cases (iii), $f$ is the identity morphism as $Y$ is smooth and it does not contain rational curves.
	 In Case (iii) ($Q$-abelian surface), let $g \colon A \to Y$ be an \'etale cover of $Y$, where $A$ is an abelian variety. 
	 As the property of being an abelian variety is preserved under \'etale covers, we can suppose that $g$ is Galois and thus we conclude by \autoref{lift-quot-abelian}(b) and (c).
\end{proof}		
	
	We now prove the log liftability for general $F$-split klt Calabi--Yau surfaces. 
	
	\begin{theorem}\label{t:lift-K-trivial-klt}
		Let $X$ be a globally $F$-split projective surface with klt singularities such that $K_X \equiv 0$.
		Then $X$ is log liftable over $W(k)$.
	\end{theorem}
	 \noindent We recommend the reader to follow the diagram included in the proof while going through the argument.
	\begin{proof}
	
	By \autoref{prop:lift_minres_K0} we can suppose that $X$ has singularities worse than canonical.
	Note that this implies that $p > 2$. Indeed, if $p=2$, then a splitting of the Frobenius morphism is a non-zero section in $H^0(X, \mathcal{O}_X (-K_X ))$, and therefore $X$ is Gorenstein and thus it is has canonical singularities.
	
	Let $f \colon (Y,R) \to X$ be the minimal resolution of $X$.
	As the singularities of $X$ are worse than canonical, 
    $K_Y=f^{*}K_X-R\equiv -R$ for non-zero effective $f$-exceptional divisor $R$. Thus, 
 $h^2(Y, \mathcal{O}_Y)=h^0(Y, \mathcal{O}_Y(K_Y))=0$. Therefore by \autoref{l-loglift-surf}, it suffices to show that $(Y,R)$ admits a formal lifting over $W(k)$.
		
Since $X$ is globally $F$-split and $K_X \equiv 0$, we get $(p-1)K_X \sim 0$.
Let $d>0$ be the minimal integer such that $dK_X \sim 0$ and let $\pi \colon Z \to X$ be the canonical $d$-cyclic cover\footnote{ Precisely, $Z := \mathrm{Spec}_X\Big(\bigoplus_{i \geq 0} \mathcal{O}_X(iK_X)t^i \, /\, \mathcal{I}\Big)$, where
\[
    \mathcal{I} \coloneqq \big(t^ks - t^{k-d}\phi_k(s)\ \big | \ s \in \cO_X(kK_X) \textrm{ and } k \in \bZ_{\geq d}\big),
\]
and $\phi_k \colon \cO_X(kK_X) \xrightarrow{\cong} \cO_X((k-d)K_X)$ is induced by the fixed isomorphism $\phi_d \colon \cO_X(dK_X) \xrightarrow{\cong} \cO_X$.} (see \cite[Definition 5.19]{km-book}).

		Note that:
		\begin{enumerate}
		    \item[(a)] as $d<p$, the group scheme $\mu_{k,d}$ is multiplicative and \'etale; 
		    \item[(b)] the cover $\phi$ is quasi-\'etale as $d<p$  and there is a natural ${\mu_{d,k}}$-action on $Z$ for which $\pi$ is a $\mu_{d,k}$-torsor over codimension one points of $X$;
		    \item[(c)] $Z$ is a globally $F$-split variety by \autoref{l-F-split-quasietale} and $K_Z \sim 0$  by construction (cf.~\cite[Lemma 2.53]{km-book});
		    \item[(d)] $Z$ has  klt  singularities (hence canonical as it is Gorenstein). Indeed, as $d<p$ the morphism $\pi$ is tamely ramified everywhere and thus we can apply  the same arguments as in the proof of \cite[Proposition 5.20]{km-book}.
		\end{enumerate}
		Let $h \colon (T,E) \to Z$ be the minimal resolution. 
		Since $T$ is a minimal surface of non-negative Kodaira dimension, any birational map $T \dashrightarrow T$ is an isomorphism, and therefore $\mu_{d,k}$ acts regularly on $T$.  
		Moreover, this action is compatible with that on $Z$, and so $\mu_{d,k}$ acts regularly on the whole pair $(T,E)$.
		
		Let $g \colon (W, g_*^{-1}E+F) \to (T,E)$ be a $\mu_{d,k}$-equivariant resolution of indeterminacies of $T \dashrightarrow Y$, where $F:=\Ex(g)$ and $(W, g_*^{-1}E+F)$ is an snc pair. 
		Recall that a usual resolution of indeterminacies of rational maps between smooth surface can be constructed as a sequence of blow-ups at closed points (see \cite[Theorem II.7]{Beau96}); in our case we blow-up at $\mu_{d,k}$-orbits of closed points.

Finally, consider the quotient log pair $(U,Q):=\factor{\left(W, g_*^{-1}E+F \right)}{\mu_{d,k}}$ which fits in the following diagram:  

\[ \label{eq:diagram-of-equivariant-blow-ups}
		\xymatrix{
			(W, g_*^{-1}E+F) \ar[r] \ar[d]^{g}  & (U,Q) \ar[d]^{\phi}\\
			 (T,E) \ar@{-->}[r] \ar[d]^{h} & (Y,R) \ar[d]^{f}\\
			 Z \ar[r]^{\pi} & X.
	}
\]
In what follows we lift the above diagram over $W(k)$.  First, let $\widetilde{h} \colon (\cT_{\can}, \mathcal{E}_{\can}) \to \cZ_\can$ be the canonical lifting of $(T,E) \to Z$ over $W(k)$ constructed in \autoref{prop:lift_minres_K0}.

\begin{claim}
The $\mu_{d,k}$-action on $(T,E)$ lifts to an action of $\mu_{d, W(k)}$ on the canonical lifting $(\cT_{\can},\mathcal{E}_{\can})$.    
\end{claim}
\begin{proof}
    As $K_Z \sim 0$, then $K_T \sim 0$ and, as $p>2$, $T$ is not an Enriques surface. 
    If $T$ is hyperelliptic with $K_T \sim 0$, then by the classification of the order of the canonical class at \cite[end of page 37]{BM77}, the only possible case is when $p=3$, corresponding to case (b) in the list \cite[beginning of page 37]{BM77}, which does not appear as $T$ is globally $F$-split.
	Therefore $T$ is either a K3 or an abelian surface and we conclude that the $\mu_{d,k}$-action lifts to an action of $\mu_{d, W(k)}$ on $\cT_{\can}$ by \autoref{p-lift-Pic-K3} and \autoref{t-ms}. 
	As $\mu_{d,k}$ acts on the pair $(T,E)$, and the lifting of each irreducible component of $E$ is unique in $\cT_{\can}$, we conclude that $\mu_{d,W(k)}$ acts on $(\cT_{\can},\mathcal{E}_{\can})$.  
\end{proof}

Now, by Lemma \ref{lem-equiv-lift}, there exists a $\mu_{d, W(k)}$-equivariant birational morphism \[\widetilde{g} \colon (\mathcal{W}, \widetilde{g}_*^{-1}\mathcal{E}_{\can}+\mathcal{F}) \to (\cT_{\can}, \mathcal{E}_{\can})\] lifting $g$.

Next, let $(\mathcal{U}, \mathcal{Q})$ be the quotient of  $(\mathcal{W}, \widetilde{g}_*^{-1}\mathcal{E}_{\can}+\mathcal{F})$ by $\mu_{d, W(k)}$:
\[
	(\mathcal{W}, \widetilde{g}_*^{-1}\mathcal{E}_{\can}+\mathcal{F})  \to (\mathcal{U},\mathcal{Q}) .
\]
Clearly $(\mathcal{U}, \mathcal{Q}) $ is a lifting of $(U,Q)$.
To show that $(Y,R)$ lifts (and so $X$ is log liftable), we shall apply \autoref{t-cynkvanstraten} to construct a lift $\widetilde{\phi} \colon (\mathcal{U}, \mathcal{Q}) \to (\mathcal{Y}, \mathcal{R})$ of $\varphi \colon (U,Q) \to (Y, R)$. To this end, we need to verify that $\varphi$ satisfies the hypotheses (a) and (b) of \autoref{t-cynkvanstraten}.	
		
		As $Y$ is normal and $\varphi$ is a proper birational morphism we deduce that $\varphi_*\mathcal{O}_U=\mathcal{O}_{Y}$. 
		As $d<p$ and $W$ is smooth, the same proof as in \cite[Proposition 5.13]{km-book} yields that $U$ has rational singularities. Since $Y$ is smooth we deduce therefore that $R^1\varphi_* \mathcal{O}_U=0$. 
		Similarly, for each component $\Gamma \cong \mathbb{P}^1$ of $Q$ one can show that $R^1\varphi_*\mathcal{O}_\Gamma=0$ and $\varphi_*\mathcal{O}_{\Gamma}=\mathcal{O}_{\varphi(\Gamma)}$.
		We can thus apply \autoref{t-cynkvanstraten} repeatedly to deduce that $\phi \colon (U,Q) \to (Y,R)$ admits a formal lifting over $W(k)$.
	\end{proof}
	
The above proof used the following essential lemma.
	
	\begin{lemma} \label{lem-equiv-lift}
	Let $(T,E)$ be a smooth snc surface pair over $k$ admitting an action of $\mu_{d,k}$. 
	Let $g \colon (W, g^{-1}_*E + F) \to (T,E)$ be a $\mu_{d,k}$-equivariant birational morphism such that $(W, g^{-1}_*E + F)$ is simple normal crossing, where $F := \mathrm{Exc}(g)$. 
	
	Let $(\mathcal{T},\mathcal{E})$ be a $\mu_{d, W(k)}$-equivariant lift of $(T,E)$ over $W(k)$. 
	Then, there exists  a $\mu_{d, W(k)}$-equivariant birational morphism 
	\[
	\widetilde{g} \colon (\mathcal{W}, \widetilde{g}_*^{-1}\mathcal{E}+\mathcal{F}) \to (\cT, \mathcal{E})
	\]
	lifting $g \colon (W, g^{-1}_*E + F) \to (T,E)$.
	\end{lemma}
	
	\noindent In the proof below, we let $E = \sum E_i$ and $\mathcal{E} = \sum \mathcal{E}_i$ to be the corresponding decompositions in prime divisors. 
	We also set $E_{i_1,\ldots, i_s} \coloneqq E_{i_1} \cap \ldots E_{i_s}$ and $\mathcal{E}_{i_1,\ldots, i_s} \coloneqq \mathcal{E}_{i_1} \cap \ldots \mathcal{E}_{i_s}$.
	
	\begin{proof}	
	By induction on the number of blow-ups at closed points, it is enough to show the claim in the case of a single blow-up at a $\mu_{d,k}$-orbit $\sigma=\left\{p_1, \dots, p_r \right\}$, where $p_i$ are closed points of $T$. 
	Set $p:= p_1$, let $ H \subset \mu_{d,k}$ be the stabiliser of $p$, and let $\mathcal{H}\subset \mu_{d,W(k)}$ be the natural lifting to $W(k)$. 
			\begin{claim} There exists a smooth lifting $\widetilde{p} \subseteq \mathcal{T}$ of $p$  such that
			\begin{enumerate}
			    \item[(a)] $\widetilde{p}$ is compatible with the snc structure of $(\mathcal{T},\mathcal{E})$ (see \cite[Definition 2.7]{ABL20}), and
			    \item[(b)] $\widetilde{p}$ is $\mathcal{H}$-invariant.
			\end{enumerate}
			\end{claim}
			
			\begin{proof}
			Suppose that $p$ lies in the smooth stratum $E_{i_1,\cdots, i_s}$ and no smaller one. Since $H$ stabilises $p$, we must have that $H(E_{i_1,\cdots, i_s}) = E_{i_1,\cdots, i_s}$. Since $\mathcal{H}$ acts on $\mathcal{E}$, this implies that $\mathcal{H}(\mathcal{E}_{i_1,\cdots, i_s}) = \mathcal{E}_{i_1,\cdots, i_s}$
			
			Let $S \subseteq E_{i_1,\cdots, i_s}$ be the fixed locus of the action of $H$ on $E_{i_1,\cdots, i_s}$. 
			As the geometric fibres of $\mathcal{H} \to \Spec(W(k))$ are linearly reductive, the fixed locus $\mathcal{S}\subseteq \mathcal{E}_{i_1,\ldots, i_s}$ of the action of $\mathcal{H}$ on $\mathcal{E}_{i_1,\ldots, i_s}$  is smooth over $W(k)$ by \cite[Proposition A.8.10]{CGP15}.
 			As $\mathcal{S}$ is smooth, we can choose a lifting $\widetilde{p}$ of $p$ inside $\mathcal{S}$ (see \cite[The\'or\'eme 18.5.17]{Gr67}, cf.\ \cite[Lemma 2.8]{ABL20}), which satisfies (1) and (2) by construction. 
\end{proof}
		    
	As $k$ is algebraically closed, $\mu_{d,k}(k) \neq \emptyset$. As $d$ is coprime to $p$, we have $\mu_{d,k}(k) \cong \mathbb{Z}/d\mathbb{Z}$ as groups and by Hensel's lemma we deduce also that $\mu_{d,W(k)}(W(k)) \cong \mathbb{Z}/d\mathbb{Z}$.
	
	Set $\Sigma$ to be the orbit of $\tilde{p}$ obtained by the action of $(\mathbb{Z}/d\mathbb{Z})$ and let $\Sigma_1$ be the connected component of $\Sigma$ containing $p_1$. 
	As every irreducible component $\Theta$ of $\Sigma$ is a section of $\mathcal{T} \to \Spec(W(k))$ by construction, and every irreducible component of $\Sigma_1$ passes through $p$ we deduce that $p_i \notin \Sigma_1$ for $i \neq 1$.
	
	\begin{claim} $\Sigma_1 = \tilde{p}$.
	\end{claim}
	Recall that $\tilde{p}$ restricted to the central fibre is $p_1$.
	
\begin{proof}
	Note that
    \[
		\Sigma = \bigcup_{g \in \mathbb{Z}/d\mathbb{Z}} g(\tilde{p}) = \bigcup_{g \in (\mathbb{Z}/l\mathbb{Z})} g(\tilde{p}),
	\]
	where the last equality follows from the fact  $\mathcal{H}$ stabilises $\widetilde{p}$.
	If $\mathcal{H}(W(k)) \cong l\mathbb{Z}/d\mathbb{Z}$ for some $l>0$, then $(\mu_{d,W(k)}/\mathcal{H})((W(k)))  \cong \mathbb{Z}/l\mathbb{Z}$.
	A simple counting shows that the closed subschemes
	\[
		\big(g(\tilde{p}) \subseteq \mathcal{T} \ \big|\ g \in (\mathbb{Z}/l\mathbb{Z}) \big)
	\]
	specialise to distinct points (one of $p_1, \ldots, p_r$), and so these are all disjoint closed subschemes of $\Sigma$, that is 
	\[
		\Sigma = \bigsqcup_{g \in (\mathbb{Z}/l\mathbb{Z})} g(\tilde{p}).
	\]
			
	In particular, $(\Sigma_1)$ is a disjoint union of some of these closed subschemes, but since $p_i \not \in \Sigma_1$ for $i \neq 1$,  we conclude that $(\Sigma_1) = \tilde{p}$.  
\end{proof}
			
			By the above claim and the $(\mathbb{Z}/d\mathbb{Z})$-symmetry we get that
			\[
			\Sigma = \bigsqcup_{1 \leq i \leq r} \tilde{p}_i,
			\]
			where $\tilde{p}_i$ is a smooth lifting of $p_i$ which is compatible with the snc structure of $(\mathcal{T},\mathcal{E})$. Thus, $\Sigma$ is smooth over $W(k)$ and the blow-up along $\Sigma$ gives the desired lifting as in the proof of \cite[Proposition 2.9]{ABL20}.
\end{proof}

\subsection{$K_X \not \equiv 0$ or  $X$ is not klt}\label{s-not-pseff}
	So far we proved log liftability of $X$ when $K_X \equiv 0$ and $X$ is klt (see \autoref{t:lift-K-trivial-klt}). In this subsection, we cover the remaining cases. Note that we will repeatedly use \autoref{l-Q-K+delta} without mentioning it.

	We begin by studying globally $F$-split surface pairs admitting a Mori fibre space structure.
	
\begin{proposition}\label{LCTF2}
	Let $(X, D)$ be a  globally $F$-split projective surface pair such that $D$ is a reduced  Weil divisor. Let $f\colon X\to Z$ be a projective morphism such that 
	\begin{enumerate}
		\item[(a)] $f_{*}\MO_X=\MO_Z$ and $\dim\,Z=1$,
		\item[(b)] $-(K_X+D)$ is $f$-nef and $-K_X$ is $f$-ample.
	\end{enumerate} 
	Then $H^2(X, T_X(-\log D))=0$. Moreover, $(X, D)$ is log liftable over $W(k)$.
\end{proposition}
 \begin{proof}We first show that $H^2(X, T_X(-\log D))=0$.
        By Serre's duality, the desired vanishing is equivalent to
        \[
        H^0(X, (\Omega^{[1]}_X(\log D) \otimes \omega_X)^{**})=0.
        \]
        In particular, it suffices to show that $f_{*}(\Omega^{[1]}_X(\log D) \otimes \omega_X)^{**}=0$.
        Since this sheaf is torsion-free, it suffices to show that this is of rank zero.
        Thus, the assertion is local on $Z$, and we can shrink $Z$ if necessary.
        
        By shrinking $Z$, we may assume that $Z$ is affine, $(X,D)$ is log smooth.
        By \cite[Proposition 5.7]{Eji19} shows that $(F, D|_{F})$ is globally $F$-split.
        As $(F, D|_F)$ is globally $F$-split we deduce $D|_F$ is zero, a point, or two distinct points. 
		In particular, the pair $(F, D|_F)$ is snc.

        Since $Z$ is affine and $(X,D)$ is log smooth, we have
        \[
        f_{*}(\Omega^{[1]}_X(\log D) \otimes \omega_X)^{**}=
        H^0(X, \Omega_X(\log\,D)\otimes \omega_X).
        \]
        We show that the latter cohomology vanishes.
        Suppose by contradiction, we assume that there exists an injective $\mathcal{O}_X$-module homomorphism $\omega_X^* \hookrightarrow \Omega_X^{1}(\log D)$.
    	We now follow the proof of \cite[Lemma 4.11]{Kaw21}.
    	It is easy to see that the composition $\omega_X^{*} \to \Omega^1_X(\log D) \to \Omega_{X/Z}(\log D)$ is zero as otherwise the following chain of inequalities $2=\deg(\omega_X^{*}|_F) \leq (K_X+D) \cdot F \leq 0$ hold. 
		Therefore there is an induced injective homomorphism $\omega_X^{*} \to f^*\omega_Z \to \Omega^1_X(\log D)$, but this contradicts with $2=\deg(\omega_X^{*}|_F)$ and $f^*\omega_Z \cdot F=0$.
        Therefore, we conclude that $H^2(X, T_X(-\log D))=0$.

        Since $-K_X$ is $f$-ample, $H^0(X,\mathcal{O}_X(K_X))=0$ and by Serre duality $H^2(X, \MO_X)=0$.
		
		Let $\pi \colon Y\to X$ be a log resolution of $(X, D)$, $E\coloneqq \Ex(\pi)$ and $D'\coloneqq \pi_{*}^{-1}D$.
		By \cite[Remark 4.2]{Kaw21} we have an injection $H^2(Y, T_Y(-\log\,(D'+E)))\hookrightarrow H^2(X, T_X(-\log\,D))=0$. Since \[H^0(Y, \mathcal{O}_Y(K_Y))\hookrightarrow H^0(X, \mathcal{O}_X(K_X))=0,\] we conclude by Serre duality that $H^2(Y,\cO_Y)=0$. Therefore $(Y, D'+E)$ lifts over $W(k)$.
	\end{proof}

	\begin{proposition}\label{l_del_pezzo_case}
		Let $(X, D)$ be a globally $F$-split projective surface pair such that $D$ is a reduced  Weil divisor.
		Suppose that $X$ is a klt del Pezzo surface of Picard rank $\rho(X)=1$.
		Then there exists a log resolution $g \colon Z \to X$ of $(X,D)$ such that \[H^2(Z, T_Z(-\log (g_*^{-1}D+\Ex(g))))=0.\]
		In particular, $(X, D)$ is log liftable over $W(k)$.
	\end{proposition}
	
	\begin{proof}
In what follows we will construct a log resolution $h \colon Z \to X$ of $(X,D)$ such that the $\mathbb Q$-divisor $D_Z \geq 0$, where $K_Z+D_Z = h^*(K_X+D)$.
		Fix $\frac{1}{2}<\varepsilon<1$.
		Since $(X, D)$ is globally $F$-split, $(X,D)$ is log canonical and $-(K_X+D)$ is $\mathbb{Q}$-effective. 
		Since $X$ a klt del Pezzo with $\rho(X)=1$ we thus conclude that the pair $(X, \varepsilon D)$ is log del Pezzo.
		By \cite[Theorem 2.31]{kk-singbook}, the components of $D$ are regular or nodal.
		Let $D_1$ be the union of all nodal curves in $D$ and $D_2\coloneqq D-D_1$.
		
		Let $\pi\colon Y\to X$ be the minimal resolution of $X$ with $E\coloneqq \Ex(\pi)=\sum_i E_i$. Then we have 
		\[
		K_Y+\pi_{*}^{-1}\varepsilon D_1+\pi_{*}^{-1}\varepsilon D_2+\sum_{i=1}^n a_iE_i=\pi^{*}(K_X+\varepsilon D)		\]
		for some $0 \leq a_i<1$. We note that outside the nodes of the irreducible components of $\pi^{-1}_*D_1$ the morphism $\pi$ is a log resolution of $(X,D)$.
  
		Next, let $f\colon Z\to Y$  be the blow-up of all nodal points of $\pi_{*}^{-1}D_1$, $F\coloneqq \Ex(f)$, and $g=f\circ \pi$.
		Then we have 
		\[
		K_Z+g_{*}^{-1}\varepsilon D_1+\sum f^{-1}_{*}a_iE_i+g_{*}^{-1}\varepsilon D_2+(2\varepsilon-1)F=g^{*}(K_X+\varepsilon D).
		\]
		Note that $\Supp(g_{*}^{-1}\varepsilon D_1+\sum f^{-1}_{*}a_iE_i+g_{*}^{-1}\varepsilon D_2+(2\varepsilon-1)F)$ is snc and there exists an effective $g$-exceptional and $g$-anti-ample $\mathbb{Q}$-divisor $G$ on $Y$.
		Thus for $0<\delta \ll 1$, the pair $(Z, g_{*}^{-1}\varepsilon D_1+\sum f^{-1}_{*}a_iE_i+g_{*}^{-1}\varepsilon D_2+(2\varepsilon-1)F+\delta G)$ is log del Pezzo and we can conclude the desired vanishing and the lifting over $W(k)$ by \autoref{p-lift-Fanopairs}.
	\end{proof}
	
\begin{remark}
    The statements of \autoref{LCTF2} and \autoref{l_del_pezzo_case} might look quite technical at first due to presence of a  reduced Weil divisor $D$. 
    However, including $D$ allows to prove log liftability of globally $F$-split surfaces with log canonical singularities as shown in the proof of \autoref{t-notklt-notk0}(b).
\end{remark}
	
	\begin{theorem}\label{t-notklt-notk0}
		Let $(X,D)$ be a globally $F$-split projective surface pair such that $D$ is a reduced Weil divisor.
		Suppose that one of the following holds:
		\begin{enumerate}
			\item[(a)]  $K_X+D \not \equiv 0$;
			\item[(b)] $(X,D)$ is not klt.
		\end{enumerate} 
		Then $(X, D)$ is log liftable over $W(k)$.
	\end{theorem}
	\begin{proof}
		Let $h\colon Z\to X$ be a dlt blow-up (see \cite[Definition 4.3]{Kaw21} for example) and $D_Z\coloneqq h_*^{-1}D+ \Ex(h)$.
		Then $(Z, D_Z)$ is a globally $F$-split pair by \autoref{l-GFR-pullback}.
		To prove the theorem, it is thus sufficient to show that $(Z, D_Z)$ is log liftable over $W(k)$.  By \autoref{l-Q-K+delta}, $-(K_Z+D_Z)$ is $\Q$-effective.
		
		First assume (a).  
		Since $(X,D)$ is globally $F$-split and $K_X+D \not \equiv 0$, we get that $K_X+D$ (and so $K_Z+D_Z$) are not pseudo-effective. By running a $(K_Z+D_Z)$-MMP we obtain a birational contraction $\phi \colon Z\to W$, where the pair $(W, D_W:=\phi_*D_Z)$ is dlt, it admits a Mori fibre space structure and it is globally $F$-split by \autoref{l-gfs-image}.
		By \autoref{l-loglifting-image} it suffices to show that $(W, D_W)$ is log liftable over $W(k)$.
		If $(W, D_W)$ is a Mori fibre space to a curve, then the assertion follows from \autoref{LCTF2}.
		If $(W, D_W)$ is a Mori fibre space to a point, then $W$ is a klt del Pezzo surface of Picard rank one and thus we apply \autoref{l_del_pezzo_case}.
		
		Next we assume that $K_X+D \equiv 0$ and that (b) holds. 
		In particular, $K_Z+D_Z\sim_{\Q}0$ by \autoref{l-Q-K+delta} and, as $(X,D)$ is not klt, $D_Z\neq 0$. 
		Hence $K_Z$ is not pseudo-effective. In this case we run a $K_Z$-MMP $\varphi\colon Z\to W$ and set $D_W := \phi_*D_Z$.
		Since $K_Z+D_Z\equiv0$, the negativity lemma shows that $K_Z+D_Z\equiv \varphi^{*}(K_W+D_W)$. 
		Thus, it follows that $W$ is a Mori fibre space with klt singularities, $(W, D_W)$ is a globally $F$-split surface pair by \autoref{l-gfs-image}, and $K_W+D_W\equiv 0$. 
		Then, by  \autoref{LCTF2} and \autoref{l_del_pezzo_case}, we conclude that $(W, D_W)$ is log liftable over $W(k)$ and so is $(Z, D)$ by \autoref{l-loglifting-image}.
	\end{proof}

	We are now ready to prove log liftability of globally $F$-split surfaces.
	
	\begin{theorem}\label{t-log-lift-GFS}
	    Let $(X,D)$ be a globally $F$-split surface pair, where $D$ is a reduced Weil divisor. 
	    Then $(X,D)$ is log liftable over $W(k)$.
	\end{theorem}
	
	\begin{proof}
	By \autoref{t-notklt-notk0}, we may assume that $(X,D)$ is klt and $K_X+D \equiv 0$. Since $(X,D)$ is klt and $D$ is reduced, we have $D=0$ and $X$ is a klt Calabi--Yau surface. Thus, we can conclude by \autoref{t:lift-K-trivial-klt}. 
	\end{proof}
	
	\section{Liftability of globally $F$-split surfaces} \label{s-lift-GFS}
	
	In the previous section we showed that, given a normal globally $F$-split variety $X$ and a log resolution $f \colon Y \to X$, the pair $(Y, \mathrm{Exc}(f))$ lifts over the ring of Witt vectors. 
	In this section we will show that $X$ itself lifts and, in fact, also the morphism $f$ lifts over $W(k)$. 
	Note that the existence of a formal lifting of $X$ is automatic from log liftability when $X$ has rational singularities after \autoref{t-cynkvanstraten}. 
	We are thus left to address two main problems in this section:
	\begin{enumerate}
	    \item[(i)] prove algebraisation of some of the formal liftings of globally $F$-split surfaces with rational singularities;
	    \item[(ii)] construct a lifting when $X$ has singularities worse than rational. 
	\end{enumerate}
	In \autoref{ss-lift-dlt} and \autoref{ss-extension}, we prove the algebraisation of a formal lifting and deduce the liftability of $f$, except in the case where $X$ has strictly log canonical singularities and $K_X \sim 0$. 
	In this latter case, it is not true in general that \emph{every} lift of $Y$ descends to a lift of $X$ and we need to pick the lift of $Y$ in a intelligent way.
	This last case occupies \autoref{ss-lc-CY}, where we combine birational geometry considerations with the construction of canonical liftings for log smooth log Calabi--Yau pairs.
	
	\subsection{Liftability of dlt models}\label{ss-lift-dlt}
	
	In this subsection, we prove the existence of a projective lifting of dlt modifications of globally $F$-split surface pairs. 
	We start with the case of Calabi--Yau surfaces with canonical singularities.
	
	\begin{proposition}\label{p-algebraicity-k-trivial}
		Let $X$ be a globally $F$-split surface with canonical singularities such that $K_X\sim 0$. Let $f \colon (Y,E) \to X$ be the minimal resolution and let $(\mathcal{Y}_\can, \mathcal{E}_\can)$ be the canonical lifting of \autoref{prop:lift_minres_K0}. Then there exists a projective birational morphism \[\widetilde{f} \colon (\mathcal{Y}_\can, \mathcal{E}_\can) \to \mathcal{X}_{\can}\] of projective varieties lifting $f$ over $W(k)$.  
	\end{proposition}
	
	\begin{proof}
		Let $A$ be a very ample line bundle on $X$ and let $A_Y\coloneqq f^*A$. Let $\mathcal{A}_{\mathcal{Y}_\can}$ be a lifting of $A_Y$, whose existence is guaranteed by \autoref{prop:lift_minres_K0}. 
		As canonical surface singularities are rational \cite[Proposition 2.28]{kk-singbook}, $H^i(Y,A_Y)=H^i(X, A)=0$ for $i>0$. Therefore 
		by Grauert's theorem \cite[Corollary III.12.9]{Ha77} we have the surjectivity of the restriction map $H^0(\mathcal{Y}_\can, \mathcal{A}_{\mathcal{Y}_\can}) \to H^0(Y, A_Y)$. Therefore $\mathcal{A}_{\mathcal{Y}_\can}$ is base point free and the induced morphism $\widetilde{f}$ is a lifting of $f$.
	\end{proof}
	
	To deal with the remaining case we need the following.
	
	\begin{proposition}\label{p-extending-algebraisable}
		Let $(X,D)$ be a normal projective surface with rational singularities where $D$ is reduced and let $f\colon (Y, f_*^{-1}D+E) \to (X,D)$ be a log resolution.
		Suppose there exists a lifting $(\mathcal{Y}, \mathcal{D}_Y+\mathcal{E})$ of $(Y, D+E)$ over $W(k)$.
		Then there exists a lifting $\tilde{f} \colon (\mathfrak{Y}, \mathfrak{D}_Y+\mathfrak{E}) \to (\mathfrak{X}, \mathfrak{D})$ of $f$ in the category of formal schemes over $\Spf(W(k))$.
		If $H^2(Y, \mathcal{O}_Y)=0$, then $\widetilde{f}$ is algebraisable.
	\end{proposition}
	\begin{proof}
		Write $E\coloneqq \sum_{i} E_i$ and $\mathcal{E}=\sum_{i} \mathcal{E}_i$, where each $E_i$ is an irreducible component of $E$ and each $\mathcal{E}_{i}$ is a lifting of $E_i$.
		Since $R^1f_{*}\MO_Y=0$ and $R^1f_{*}\mathcal{O}_{D_{i}}=0$ for each irreducible component $D_i$ of $f_{*}^{-1}D$, an iterated use of \autoref{t-cynkvanstraten} shows the existence of the formal lifting $\widetilde{f}$ of $f$ over $\Spf(W(k))$.
		
		Suppose $H^2(Y, \mathcal{O}_Y)=0$ and let $A$ be an ample line bundle on $X$ and let $A_Y=f^*A$. By \cite[Corollary 8.5.6]{FAG} $A_Y$ lifts to a big and nef line bundle $\mathcal{A}_{\mathcal{Y}}$ on $\mathcal{Y}$. As $Y$ has rational singularities, $H^i(Y, A_Y^{\otimes m})=H^i(X, A^{\otimes m})$ and for $i>0$ it vanishes by Serre vanishing for sufficiently large $m$. Therefore by semicontinuity $H^i(Y_K, \mathcal{A}_{Y_K}^{\otimes m}) =0 $ for $i>0$ and by Grauert's theorem \cite[Corollary III 12.9]{Ha77} we conclude the surjectivity of $H^0(\mathcal{Y}, \mathcal{A}_\mathcal{Y}^{\otimes m}) \to H^0(Y,  A_Y^{\otimes m})$. The morphism associated to $\mathcal{A}^{\otimes m}_{\mathcal{Y}}$ is the algebraisation of $\widetilde{f}$.
	\end{proof}
	
	With the previous results, we can finally prove that dlt modifications of globally $F$-split pairs lift over $W(k)$.
	
	\begin{theorem}\label{c-lift-min-dlt}
		Let $(X,D)$ be a globally $F$-split surface pair where $D$ is reduced.
		Let $f \colon (Y, \pi_*^{-1}D+\Ex(f)) \to (X,D)$ be a dlt modification.
		Then every log resolution \[g \colon (Z, g_*^{-1}(\pi_*^{-1}D+\Ex(f))+\Ex(g)) \to (Y, \pi_*^{-1}D+\Ex(f))\] lifts to 
		$\widetilde{g} \colon (\mathcal{Z}, g_*^{-1}(\pi_*^{-1}\mathcal{D}+\mathcal{E})+\mathcal{G}) \to (\mathcal{Y}, \pi_*^{-1}\mathcal{D}+\mathcal{E})$ over $W(k)$.
	\end{theorem}

	\begin{proof}
		Recall that dlt surface singularities are rational by \cite[Proposition 2.28]{kk-singbook}.
		If $H^2(Y, \mathcal{O}_Y)=0$, we conclude by \autoref{t-log-lift-GFS} and \autoref{p-extending-algebraisable}.
		If $H^2(Y, \mathcal{O}_Y) \neq 0$, then $H^0(Y,\mathcal{O}_Y(K_Y)) \neq 0$ by Serre duality and thus $D=0$ and $X$ is a Calabi--Yau with canonical singularities, so we conclude by \autoref{p-algebraicity-k-trivial}.
	\end{proof}	
	
	\subsection{An extension theorem} \label{ss-extension}
	
	We prove an extension theorem for sections of big and nef line bundles on a dlt modification.
	The whole point of the following proposition is that, in general, a lift of a big and semiample line bundle  need not be semiample. However, if we assume that the lift of the line bundle stays trivial on the lift of the exceptional locus, then semiampleness does indeed extend.

	\begin{proposition}\label{p-lifting-criterion}
		Let $(X,D)$ be a log canonical projective surface pair, where $D$ is reduced and let $f \colon (Y, D_Y+E) \to (X,D)$ be a dlt modification, where $E$ is the reduced exceptional divisor and $D_Y:=f^{-1}_*D$.
		Let $A$ be a line bundle on $X$ and let $A_Y:=f^*A$.
		Suppose there exists a projective lifting $(\mathcal{Y}, \mathcal{D}_{\mathcal{Y}}+\mathcal{E})$ of $(Y, D_Y+E)$ over $W(k)$ together with a lifting $\mathcal{A}_{\mathcal{Y}}$ of $A_Y$.
		If $\mathcal{A}_{\mathcal{Y}}|_{\mathcal{E}}\sim 0$, then
		\begin{equation}\label{eq:surj}
		    H^0(\mathcal{Y}, \mathcal{A}_{\mathcal{Y}} ^{\otimes m}) \to H^0(Y, A_Y ^{\otimes m})
		\end{equation}
		is surjective for sufficiently large $m>0$.
		In particular, $\mathcal{A}_{\mathcal{Y}}$ is semi-ample and it induces a lifting $\widetilde{f} \colon (\mathcal{Y}, \mathcal{D}_\mathcal{Y}+ \mathcal{E}) \to (\mathcal{X}, \mathcal{D})$ of $f$ over $W(k)$.
	\end{proposition}
	
	\begin{proof}
		As $A_Y$ is nef, then $\mathcal{A}_{\mathcal{Y}_K}$ is also nef. As $\mathcal{A}_{\mathcal{Y}_K}^2=A_Y^2>0$, we conclude that $\mathcal{A}_{\mathcal{Y}}$ is a big and nef line bundle on $\mathcal{Y}$.
		To show the desired surjectivity, by projection formula it is sufficient to prove that 
		\[H^1(\mathcal{Y}, \mathcal{A}_{\mathcal{Y}} ^{\otimes m}(-Y)) \cong H^1(\mathcal{Y},  \mathcal{A}_\mathcal{Y} ^{\otimes m}) \otimes (p) \to H^1(\mathcal{Y},  \mathcal{A}_\mathcal{Y} ^{\otimes m})\]
		is injective. To this end, it is enough to check that $H^1(\mathcal{Y}, \mathcal{A}_\mathcal{Y} ^{\otimes m})$ is a free $W(k)$-module.
		By Grauert's theorem \cite[Corollary III.12.9]{Ha77} we only have to verify that the dimensions of the cohomology groups $H^1(\mathcal{Y}_s, \mathcal{A}_{\mathcal{Y}_s}^{\otimes m})$ for any $s \in \Spec(W(k))$ remain constant.
		Note that $R^if_*(A_Y^{\otimes m}(-E))=0$ for $i>0$ by the Grauert-Riemenschneider vanishing theorem for surfaces \cite[Theorem 10.4]{kk-singbook}, the assumptions of which are satisfied as
		\begin{enumerate}
		    \item[(a)] $A_Y^{\otimes}(-E) \cong \mathcal O_Y(K_Y + D_Y - f^*(K_X+D)) \otimes f^*A^{\otimes m}$, and
		\item[(b)] $D_Y$ is reduced; in particular, $D_Y^{<1}=0$ and $f_*^{-1}\lfloor{D\rfloor}=D_Y$.
		\end{enumerate}
		By projection formula we then deduce \[ H^i(Y,A_Y^{\otimes m}(-E))=H^i(X,f_*(A_Y^{\otimes m}(-E)))=H^i(X, A^{\otimes m} \otimes f_*\mathcal{O}_Y(-E)),\]
		is zero by Serre vanishing if $i>0$ and $m$ is sufficiently large.
		By semi-continuity of cohomology groups \cite[Theorem III.12.8]{Ha77}, we deduce \[
		H^i(\mathcal{Y}_K, \mathcal{A}_{\mathcal{Y}_K}^{\otimes m }(-\mathcal{E}_K))=0
		\]
		for $i>0$. 
		Therefore 
		\[H^1(\mathcal{Y}_s, \mathcal{A}_{\mathcal{Y}_s}^{\otimes m}) \cong H^1(\mathcal{E}_s, \mathcal{A}_{\mathcal{Y}_s}^{\otimes m}|_{\mathcal{E}_s})=H^1(\mathcal{E}_s, \mathcal{O}_{\mathcal{E}_s}),\] where the last equality follows from the hypothesis $\mathcal{A}_{\mathcal{Y}}|_{\mathcal{E}}\sim 0$. Clearly $H^1(\mathcal{E}_s, \mathcal{O}_{\mathcal{E}_s})$ is constant in a flat family of integral curves, thus concluding. 
		
		We now explain the construction of the lifting of $f$. As $A_Y$ is semi-ample, by the surjectivity of \autoref{eq:surj}, we deduce that $\mathcal{A}_\mathcal{Y}$ is also semi-ample over $W(k)$ and \[\widetilde{f}\colon (\mathcal{Y}, \mathcal{D}_Y+ \mathcal{E}) \to (\mathcal{X}\coloneqq \Proj_{W(k)}  R(\mathcal{Y}, \mathcal{A}_\mathcal{Y}), \widetilde{f}_*\mathcal{D}_{\mathcal{Y}})\] is a lifting of $f\colon (Y,E) \to (X,D)$. 
	\end{proof}
	
	The previous extension theorem, combined with the techniques of \autoref{s-def-toolbox}, allows to descend liftability over $W(k)$ from the dlt modification in several cases.
	
	\begin{corollary}\label{p-lc-nottriv}
		Let $(X, D)$ be a globally $F$-split projective surface pair, where $D$ is reduced. Let $f\colon (Y,D_Y+\Ex(f)) \to (X,D)$ be a log resolution, where $D_Y := f^{-1}_*D$.
		If $D \neq 0$ or $H^0(X, \mathcal{O}_X(K_X+D)) = 0$, then there is a lifting $f \colon (\mathcal{Y}, \mathcal{D}_{\mathcal{Y}}+\mathcal{E}) \to (\mathcal{X}, \mathcal{D})$ of $f$ over $W(k)$.
	\end{corollary}
    Observe that if $X$ is globally $F$-split, then $H^0(X,\cO_X(K_X)) \neq 0$ exactly when $K_X\sim 0$; in particular, each singularity of $X$ is either Gorenstein canonical or Gorenstein strictly log canonical. The former case was already solved, and we shall work on the latter in the next subsection.
	
	\begin{proof}
		By \autoref{c-lift-min-dlt}, we can reduce to proving the existence of a lifting $\widetilde{f}$ over $W(k)$ of a dlt modification $f \colon (Y, D_Y+E) \to (X,D)$. 
		Note that $(Y, D_Y+E)$ is globally $F$-split by \autoref{l-GFR-pullback} and let $(\mathcal{Y}, \mathcal{D}_{\mathcal{Y}}+\mathcal{E})$ be a lifting of  $(Y, D_Y+E)$ over $W(k)$ given by \autoref{c-lift-min-dlt}.
		
		Let $A$ be a very ample line bundle on $X$ and consider $A_Y:=f^*A$.
		Fix an isomorphism $\varphi \colon A_Y|_E \to \mathcal{O}_E$. 
		By \autoref{c-lift-lb-triv} the obstruction classes to the existence of a lifting $(\mathcal{A}_{\mathcal{Y}}, \widetilde{\varphi})$ of the $E$-trivial line bundle $(A_Y, \varphi)$ lie in $H^{2}(Y, \mathcal{O}_Y(-E))$. 
		By \autoref{p-lifting-criterion} it is sufficient to show that $H^{2}(Y, \mathcal{O}_Y(-E))=0$.
		
		If $H^0(X, \mathcal{O}_X(K_X+D)) = 0$, then $H^0(Y, \mathcal{O}_Y(K_Y+E+D_Y))=0$. Since we have \[H^0(Y, \mathcal{O}_Y(K_Y+E)) \subset H^0(Y, \mathcal{O}_Y(K_Y+E+D_Y)),\] we conclude.
		If $D\neq 0$ and $H^0(Y, \mathcal{O}_Y(K_Y+E+D_Y)) \neq 0$, then $K_Y+E+D_Y \sim 0$ by \autoref{l-Q-K+delta}, and therefore $H^0(Y, \mathcal{O}_Y(K_Y+E))=0.$
	\end{proof}

	\subsection{Canonical liftings of $K$-trivial surfaces} \label{ss-lc-CY}
	
	We prove the liftability of globally $F$-split $K$-trivial varieties with strictly log canonical singularities by constructing a `canonical' log lifting. 
	We give the following example in which, if the lifting $(\mathcal{Y}, \mathcal{E})$ of the minimal resolution $f\colon (Y,E) \to X$ is chosen generically, then $f$ does not lift.
	
	\begin{example}\label{e-nagata}
		We fix $k=\overline{\mathbb{F}}_p$ to be the algebraic closure of $\mathbb{F}_p$
		and let $E \subset \mathbb{P}^2_k$ be a globally $F$-split elliptic curve.
		Choose $P_1, \dots, P_9 \in E$ distinct points in general position and
		let $h \colon Y\to \mathbb{P}^2_k$ be the blow-up at these points. The pair $(Y,E_Y\coloneqq h_*^{-1}E)$ is globally $F$-split by \autoref{l-GFR-pullback} and by \cite[Corollary 0.3]{Kee99}, there is a birational contraction $f \colon Y\to X$ contracting $E_Y$. In particular, $X$ is a log canonical surface with $K_X \sim 0$.
		Choose a lifting $(\mathbb{P}^2_{W(k)}, \mathcal{E})$ of $(\mathbb{P}^2_k,E)$ together with liftings $\mathcal{P}_i \subset \mathcal{E}$ of $P_i$. We denote by $K$ the fraction field of $W(k)$.
		By blowing-up $\mathcal{P}_i$, we construct a lifting $(\mathcal{Y},\mathcal{E})$ of $(Y,E_Y)$ over $W(k)$.
		However if the points $\mathcal{P}_{i,K}$ are in general position in $\mathcal{E}_{K}$, then we cannot expect to find a birational contraction of $\mathcal{E}$ as explained in \cite[Example V.5.7.3]{Ha77}.  
	\end{example}
	
	In particular, we cannot prove liftability of $f$ or $X$ as a direct consequence of \autoref{t-log-lift-GFS}.
	To solve this problem, we turn the presence of non-klt singularities to our advantage by constructing a well-chosen lifting of a crepant resolution.
	For this we begin by studying their crepant snc birational models.
	We will repeatedly use the following remark on factorisation of crepant birational maps of smooth surfaces.
	
	\begin{lemma} \label{lem: crepant-correct}
	   Let $\psi \colon (X,D_X) \dashrightarrow (Z, D_Z)$ be a crepant birational map of surface pairs with reduced boundaries.
	   Suppose $(X, D_X)$ is an snc pair.
	   Then there exists a commutative diagram	
	   \begin{equation*}
	       	   \begin{tikzcd}
                   & (W,D_W) \arrow[dl,swap,"f"] \arrow[dr,"g"] &       \\
        (X,D_X) \arrow[rr, dotted, swap, "\psi"]  &              & (Z,D_Z),
        \end{tikzcd}
        \end{equation*}
        where $(W,D_W)$ is an snc pair with reduced boundary and $f$ and $g$ are crepant birational.
	\end{lemma}

\begin{proof}
 Let $(W, D_W)$ be the \emph{minimal} resolution of indeterminacies of $\psi$ that is obtained by subsequently blowing-up along the points at which $\psi$ is not defined.
 We show $D_W$ is effective.
 Suppose by contradiction that there exists an irreducible component $E$ of $D_W$ such that $\text{coeff}_{E}(D_W)<0$. Then $E$ is $f$-exceptional and the centre $P\coloneqq f(E)$ is not contained in $D_X$. As $D_Z$ is effective, we conclude that $f^{-1}(P)$ is contracted by $g$. Therefore, $P$ is contained in the locus where $\psi$ is defined, contradicting the minimality of $f$.
\end{proof}
	
	We recall the crepant birational classification of log Calabi--Yau structures on minimal rational surfaces.
	
	\begin{lemma}\label{l-crepant-model-logCY-hirzebruch}
     Let $X$ be the projective plane $\mathbb{P}^2$ or a Hirzebruch surface $\mathbb{F}_n$ for $n \in \mathbb{N}$. Further, let $D$ be a reduced Weil divisor such that $(X,D)$ is a log Calabi--Yau pair.
		Then $(X,D)$ is crepant birational to one of the following: 
		\begin{enumerate}
		    \item[(a)] $(\mathbb{P}^2_{k}, E)$, where $E$ is an elliptic curve;
		    \item[(b)] $(\mathbb{P}^2_{k}, L_1+L_2+L_3)$, where $L_i$ are lines in general position. 
		\end{enumerate}
	\end{lemma}

	\begin{proof}
	This result is well-known, we recall a proof for completeness.
	We start by reducing the problem to the study of log Calabi--Yau pairs on $\mathbb{P}^2$. Suppose $X \cong \mathbb{F}_n$.
	If $n=1$, then there exists a crepant birational contraction $\pi \colon (\mathbb{F}_1, D_1) \to (\mathbb{P}^2_{k},E=\pi_*D_1)$.
	Let $n > 1$ and denote by $C_n$ the $(-n)$-negative section of $\mathbb{F}_n$. 
	Note that $D_n \neq C_n$: otherwise, $K_{\mathbb{F}_n}+C_n \sim 0$ and by adjunction $(K_{\mathbb{F}_n}+C_n) \cdot C_n=-2$, which is a contradiction.
	
	We choose a smooth point $x \in D_n \setminus C_n$ belonging to a fibre $F$ of $\mathbb{F}_n \to \mathbb{P}^1$. Let $g \colon X \to (\mathbb{F}_n, D_n)$ be the blow-up at $x$ and write $K_X+\Gamma = g^*(K_{\mathbb{F}_n}+D_n) \sim 0$. 
		If $h \colon X \to \mathbb{F}_{n-1}$ is the contraction of $g_*^{-1}F$, then $(\mathbb{F}_{n-1}, D_{n-1}:=h_*\Gamma)$ is a crepant model of $(\mathbb{F}_{n}, D_{n})$ and thus we conclude by descending induction.
		In the case of $(\mathbb{F}_0, D_0)$, we blow-up a closed point $p$ on the smooth locus of $D_0$ and contract the strict transform of the fibre passing through $p$, thus ending again in $\mathbb{F}_1$.
		
		We are thus left to discuss the crepant birational models of log Calabi--Yau pairs on $(\mathbb{P}^2, E)$. In this case we use the quadratic Cremona transformations.
		As $(\mathbb{P}^2_{k}, E)$ is log canonical and $E$ is a cubic curve, then $E$ must be either an elliptic curve, the union of three lines in general position (with exactly three intersection points), the union of a line and a conic in general position (intersecting transversally at exactly two points) or a nodal curve. We show we can always reduce to the first cases.
		
		Suppose $E=C+L$ where $C$ is a conic and $L$ is a line intersecting $C$ in two distinct points. Let $p \in C \cap L$ and let $q_1, q_2 \in C \setminus L$. By applying the standard quadratic Cremona transformation with base points $p, q_1, q_2$, it is easy to see that $(\mathbb{P}^2_{k}, C+L)$ is crepant birational to $(\mathbb{P}^2_{k}, L_1+L_2+L_3)$, where $L_i$ are lines in general position. 
		Suppose $E$ is a nodal irreducible cubic curve with the node $p$. Let $q_1, q_2 \in E$ be different from $p$. 
		By applying a standard quadratic Cremona transformation at $p, q_1, q_2$, it is easy to see that $(\mathbb{P}^2_{k}, E)$ is crepant birational to  $(\mathbb{P}^2_{k}, C+L)$ where $C$ is a conic and $L$ is a line meeting in general position. This has already been proven to be crepant birational to $(\mathbb{P}^2_{k}, L_1+L_2+L_3)$.
	\end{proof}
	
	The following is a specific instance of the connectedness principle for the  non-klt-locus of pairs in the case of $K$-trivial surfaces.
	
	\begin{proposition}\label{l-logcanonical-CY}
		Let $X$ be a log canonical projective surface such that $K_X \sim 0$ and suppose that $X$ is not klt. 
		Then there exists a crepant log resolution $f \colon (Y, E) \to X$, where $K_Y+E \sim f^*K_X$, and a crepant proper birational contraction $h \colon (Y, E) \to (Z, E_Z)$ such that
		\begin{enumerate}
			\item[(i)] \label{i} $(Z, E_Z) \cong (\mathbb{P}^2_{k}, C)$ where $C$ is an elliptic curve;
			\item[(ii)] $(Z, E_Z) \cong (\mathbb{P}^2_{k}, L_1+L_2+L_3)$ where $L_i$ are three lines in general position; 
			\item[(iii)] $ (Z, E_Z) \cong (\mathbb{P}_B(M \oplus N), C+D)$, where $B$ is a curve of genus $1$, $M$ and $N$ are line bundles on $B$, and $C$ (resp.\ $D$) is the section associated to the quotient $M \oplus N \to M$ (resp.\ $M \oplus N \to N$).
		\end{enumerate}
		In particular, there are at most two non-klt points on $X$. Cases (i-ii) happen if there is exactly one non-klt point and Case (iii)  happens otherwise.
	\end{proposition}
	
	\begin{proof}
        Let $f \colon Y \to X$ be the minimal resolution, which as $K_X \sim 0$ only extracts divisors of discrepancy 0 or $-1$.
		Write $K_Y+E \sim 0$, where each coefficient of $E$ is equal to one and $E > 0$ by hypothesis. By \cite[Sections 3.39-3.40]{kk-singbook},
		the only case where $f$ is not a log resolution is if $\Ex(f)$ contains a nodal irreducible curve $D$. In this case, $D$ does not intersect any other irreducible components of $E$ and simply by blowing-up at the nodal point we reach a crepant log resolution of $X$.
		From now on, we feel free to replace $Y$ with a model obtained by blowing-up points on $E$ whenever needed.
		
	    Let $h \colon Y \to Z$ be a birational contraction induced by a $K_Y$-MMP. 
		Then $(Z, E_Z\coloneqq h_*E)$ is a log Calabi--Yau pair on a Mori fibre space $\pi \colon Z \to B$. Note that $E_Z$ has the same number of connected components as $E$: indeed, at each step of the MMP, $K_Y \cdot \xi=-1$ for an extremal ray $\xi$, which implies that $E \cdot \xi=1$, so that $\xi$ intersects $E$ only in one irreducible component. Therefore the number of non-klt singular points of $X$ is the number of connected components of $E_Z$.
		
    	If $\dim(B)=0$, then in this case $Z \cong \mathbb{P}^2_k$, $E_{\mathbb{P}^2_k} \in |-K_{\mathbb{P}^2_k}|$ and in particular $E_{\mathbb{P}^2_k}$ is connected. 
		By \autoref{l-crepant-model-logCY-hirzebruch}, there exists a crepant birational map $ \psi \colon (Y,E) \dashrightarrow (\mathbb{P}^2, D)$, where $D$ is either a smooth elliptic curve or the union of three lines in general position.
		We can now replace $(Y,E)$ with a higher birational model by applying \autoref{lem: crepant-correct}.
		
		If $\dim(B)=1$, then the N\'{e}ron-Severi group $\text{NS}(Z)=\mathbb{Z}[C_n] \oplus \mathbb{Z}[f]$ by \cite[Proposition V.2.3]{Ha77}, where $C_n^{2}=-n \leq 0$ and $C_n \cdot f=1$. 
		By adjunction, $K_Z \cdot C_n=n+\deg(K_{C_n})$.
		Note that $E_Z \sim -K_Z \sim 2C_n+bf$ for some $b \in \mathbb{Z}$. 
		 We distinguish two cases, according whether $C_n$ belongs to $\Supp(E_Z)$ or not.
		
		\emph{Case 1:  $C_n \nsubseteq \Supp(E_Z)$.} 
		Then $0 \leq E_Z\cdot C_n =-2n+b$ and
		\begin{align*}
		0=\deg(K_{E_Z})&=(K_Z+(2C_n+bf))\cdot (2C_n+bf)\\
		&=2(n+\deg(K_{C_n}))-2b-4n+4b.
		\end{align*}
		In particular, $ b=n-\deg(K_{C_n}) \text{ and } b \geq 2n. $
		This can only happen if $\deg(K_{C_n})<0$ and so $C_n \cong \mathbb{P}^1_k$. 
		In this case, $Z \cong \mathbb{F}_n$ and thus combining as before \autoref{l-crepant-model-logCY-hirzebruch} and \autoref{lem: crepant-correct} we can replace $(Z,E_Z)$ with $(\mathbb{P}^2_{k}, C)$ ending in Cases (i) or (ii).
		
		\emph{Case 2: $E_Z=C_n+D$ where $D \geq 0$.} 
		As $E_Z \cdot f=2$, we have $D \cdot f=1$ and $D \sim C_n+bf$.  In this case, there are at most two connected components of $E_Z$. 
		Note that 
		\[
		D \cdot C_n =E_Z \cdot C_n- C_n^2=-(K_Z +C_n)\cdot C_n = -\deg(K_{C_n}).
		\]
		If $\deg(K_{C_n}) =-2$, then $Z \cong (\mathbb{F}_n, E_Z)$ for some $n \geq 0$. Again by \autoref{lem: crepant-correct} and \autoref{l-crepant-model-logCY-hirzebruch}, we replace $(Z,E_Z)$ with $(\mathbb{P}^2, C)$ ending in (a).
		
		If $\deg(K_{C_n})=0$, then $B$ is an elliptic curve and, as $Z$ is a smooth surface, then $Z=\mathbb{P}_B(V)$ where $V$ is a vector bundle of rank 2. 
		As $C_n \cdot D=0$, $C:=C_n$ and $D$ are two disjoint sections of $\pi$.
		By applying \autoref{lem: split_vect_bundle} we conclude we are in case (iii).
	\end{proof}
	
	\begin{lemma}\label{lem: split_vect_bundle}
	    Let $B$ a curve over $k$ and let $V$ a vector bundle of rank 2 on $B$.
	    Then $V$ is decomposable as a sum of line bundles if and only if the projective bundle $\mathbb{P}_B(V) \to B$ has two disjoint sections.
	\end{lemma}
	
	\begin{proof}
	Let $C, D$ be two disjoint sections. By \cite[Proposition 2.6]{Ha77}, they correspond to the short exact sequences:
	\[0 \to N_C \to V \to L_C \to 0 \text{ and }
     0 \to N_D \to V \to L_D \to 0. \]
	It is easy to verify that the natural composition $N_C \to V \to L_D$ is an isomorphism if and only $C$ and $D$ are disjoint, concluding.
	\end{proof}

There is a unique way of lifting closed
points of an ordinary elliptic curve to its canonical lifting, which is compatible with the Frobenius lift, which we now recall.

\begin{lemma} \label{lem: lift_points}
    Let $E \subset \mathbb{P}^2_k$ be a globally $F$-split elliptic curve.
    Let $p \in E$. 
    Then there exists a unique lifting $\widetilde{p} \subset \mathcal{E}_\can$ of $p$ for which there is an isomorphism $\mathcal{O}_{\mathcal{E}}(\widetilde{p}) \cong \mathcal{L}_p$, where $\mathcal{L}_p \in \Pic(\mathcal{E}_{\mathcal{Z}_\can})_{F_{\mathcal{E}_\can}}$ is the canonical lifting of the line bundle $\mathcal{O}_E(p)$ of \autoref{t-ms}.
    We say that $\widetilde{p}$ is the unique lifting of $p$ compatible with the Frobenius lifting on $\mathcal{E}_{\can}$ and we call it the canonical lifting of $p$.
\end{lemma}

\begin{proof}
Note that $\deg_{\mathcal{E}_{\can,K}}(\mathcal{L}_p\vert_{\mathcal{E}_{\can,K}})=\deg_{E}(p)=1$.
By Riemann-Roch we conclude that there exists a unique  $\mathcal{L}_p=\mathcal{O}_{\mathcal{E}_{\mathcal{Z}_\can}}(\tilde{p})$ for a unique lifting $\tilde{p}$ of $p$.
\end{proof}
	
We will need the following explicit description of $\Pic^0$ on cycles of smooth rational curves.
Given a Weil divisor $D$ on a curve $C$ whose support is contained in the regular locus, we can associate a Cartier divisor $D \in H^0(C, \mathcal{K}_C^*/\mathcal{O})$ the locally free sheaf given by
\[\mathcal{O}_X(D)(U) = \left\{ f \in \mathcal{K}_C^* \mid \text{div}(f)|_U+D|_U \geq 0 \right\}\]
We refer the reader to \cite[Section 7]{Liu02} for the theory of the Picard group for non-integral curves.

The following is a generalisation of Menelaus' theorem on the collinearity for points on a triangle in elementary geometry (in the setting of Menelaus' theorem, $n=3$, $|J_i|=1$ and $d_{i1}=1$).

\begin{lemma}\label{l-Pic-cycle rational}

Let $E=E_1\cup E_2 \cup \dots \cup E_n$ be an  oriented cycle of smooth rational curves over $k$. 
Let $L\in\Pic^0(E)$ be an invertible sheaf with $L \cong \mathcal{O}_E(\sum_{i=1}^n\sum_{j \in J_i}  d_{ij} p_{ij})$
for some $d_{ij} \in \mathbb{Z}$ and some regular points $p_{ij} \in E_i$, where $J_i$ are index sets.	
In what follows, we normalise the coordinates of $E_i$ so that $E_{i} \cap E_{i+1}|_{E_i}=[1:0]$ and $E_{i} \cap E_{i+1}|_{E_{i+1}}=[0:1]$ for $1 \leq i < n$. 
Further, we write $p_{ij}=[a_{ij}:b_{ij}] \in E_i \cong \mathbb{P}^1_k$, for $a_{ij}, b_{ij} \in k^*$.
Define 		
\begin{equation}\label{formula-pic}
	\lambda(L):= \prod_{i=1}^n \prod_{j \in J_i} \left( \frac{a_{ij}}{-b_{ij}} \right)^{d_{ij}}. \tag{$\star$}
\end{equation}
Then $L \cong \mathcal{O}_X$ if and only if $\sum_{j \in J_i} d_{ij}=0$  for every $i=1, \dots, n$ and $\lambda(L)=1$. 
	\end{lemma}

	\begin{proof}
	It is immediate to see that $L$ belongs to $\Pic^0(E)$ if and only if $\sum_{j \in J_i} d_{ij}=0$ for all $i=1, \dots, n$.
		
	Let $L \in \Pic^0(E)$.
	On $E_1$, fix $f_1 \in k(t)$ such that $\text{div}(f_1) = \sum d_{1j} p_{1j}$.
	Then there exists a unique $f_2 \in k(t)$ such that $f_2([0:1])=f_1([1:0])$ and $\text{div}(f_2) =\sum d_{2j} p_{2j}$ and we construct inductively $f_l$ in this way.
	We define \[\lambda(L)=  f_1([0:1]) / f_n([1:0]). \]
		Note that the rational functions $\left\{f_i\right\}$ glue to a global (clearly trivialising) section of $L$ if and only if $\lambda(L)=1$.
		
		We are only left to unravel the formula for $\lambda$ in coordinates. 
		We fix \[f_1([x:y])= \prod_{j \in J_1} (ya_{1j}-xb_{1j})^{d_{1j}}\] as the global section of $L|_{E_1}$. Similarly, a global section for $L|_{E_2}$ must be of the form \[f_2= \mu_2 \prod_{j \in J_2} (ya_{2j}-xb_{2j})^{d_{2j}},\] for $\mu_2 \in k^*$. 
     As we demand $f_1([1:0])=f_2([0:1])$ in order to glue, we deduce that
	 $\mu_2 = \prod_{j \in J_1} (-b_{1j})^{d_{1j}} \prod_{j \in J_2} a_{2j}^{-d_{2j}}$. An inductive computation shows that $f_l$ must be defined by the formula
		\[f_l([x:y]) =  \left(\prod_{i=1}^{l-1} \prod_{j \in J_i} b_{ij}^{d_{ij}} \prod_{i=2}^{l} \prod_{j \in J_i} a_{ij}^{-d_{ij}}\right) \prod_{j \in J_l} (ya_{lj}-xb_{lj})^{d_{lj}}.\]
		As $f_1([0:1])= \prod_{j \in J_1} a_{1j}^{d_{1j}}$ we deduce \autoref{formula-pic}.
	\end{proof}
	
	Let $\omega \colon k \to W(k)$ be the Teichm\"{u}ller representative for Witt vectors.
	Note that $\omega$ preserves multiplication, but not addition. 
	We can define a morphism $\omega_{\mathbb{P}^n_k} \colon \mathbb{P}^n(k) \to \mathbb{P}^n(W(k))$ of sets such that $\omega([a_0: \dots: a_n])=[\omega(a_0): \dots :\omega(a_n)]$, which is well-defined as $\omega$ is multiplicative. 

We collected all the ingredients we need to prove the liftability of $K$-trivial surfaces with strictly log canonical singularities.
	
	\begin{theorem}\label{p-strictlyLC-case}
		Let $X$ be a projective globally $F$-split normal surface such that $X$ is not klt and $K_X \sim 0$. 
		Let $f\colon (Y, \Ex(f)) \to X$ be any log resolution such that $K_Y+E = f^*K_X$ for a reduced Weil divisor $E \subseteq \Ex(f)$. Then $f$ admits a lifting $\widetilde{f} \colon (\mathcal{Y}_{\can}, \mathcal{E}) \to \mathcal{X}$ over $W(k)$.
	\end{theorem}
	Note that $X$ has Gorenstein singularities, so they are either strictly log canonical or canonical. Explicitly, a resolution $f$ as in the statement of the theorem can be constructed as follows: at canonical singularities we take the minimal resolution which extracts a tree of $(-2)$-curves, and at strictly log canonical points we resolve by extracting an elliptic curve or a cycle of rational curves. 
	In this case, $\Ex(f) = E + F$ where $E$ is the union of exceptional divisors over strictly log canonical points and $F$ is the union of exceptional divisors over canonical points.
	\begin{proof}
    For the sake of readability, we drop the subscripts \emph{can} even though the lifts we construct will be canonical.
    Let $f \colon Y \to X$ be a log resolution such that $\Ex(f)=E+F$ and $K_Y+E = f^*K_X$.
	Up to replacing $Y$ with a higher model, we can take the contraction $h \colon (Y, E) \to (Z, E_Z)$ given by \autoref{l-logcanonical-CY}. Set $F_Z := h_*F$.
		As $X$ is not canonical, the canonical class $K_Y$ is not effective, and so  $H^2(Y, \mathcal{O}_Y)=0$. As $(Y,E)$ is globally $F$-split, so is $(Z, E_Z)$.
		In fact, $E_Z$ is also a globally $F$-split scheme\footnote{
		Indeed,  the functoriality of the trace morphisms gives a commutative diagram
		\[
		\xymatrix{
			H^0(Z, F_{*}\mathcal{O}_Z((1-p)(K_Z+E_Z)) \ar@{->>}[r]^{\qquad \qquad  \Tr_{(Z,E_Z)}} \ar[d] & H^0(Z, \mathcal{O}_Z) \ar@{->>}[d] \\
			H^0(E_Z, F_{*}\omega_{E_Z}^{(1-p)}) \ar[r]^{\quad \Tr_{E_Z}} &  H^0(E_Z, \mathcal{O}_{E_Z}).
		}
		\]
		Therefore $\Tr_{E_Z}$ is surjective and $E_Z$ is globally $F$-split.}.

		Let $A$ be a very ample line bundle on $X$ and let $A_Y=f^*A$. 
		Write \[h^*h_* A_Y = A_Y(\sum_i a_i G_i),\] where $G_i$ are the $h$-exceptional divisors and fix $L:=h_*A_Y$ on $Z$. Since $Z$ is smooth, $L$ is a line bundle.
		
		Our setting may be summarised by the following diagram:
		\[
		\begin{tikzcd}
		(Y,E) \arrow{r}{f} \arrow{d}{h} & X \\
		(Z,E_Z),
		\end{tikzcd}
		\]
		where $\Ex(f) = E+F$, $E_Z = h_*E$, and $G_i \subseteq \Supp(h)$.
		\begin{claim}\label{claim-good-lifting}
		There exists a lifting $\widetilde{h} \colon (\mathcal{Y}, \mathcal{E}) \to (\mathcal{Z}, \mathcal{E}_{\mathcal{Z}})$ of $h \colon (Y, E) \to (Z, E_Z)$ together  with liftings $\mathcal{F}$ of $F$, $\mathcal{G}_{i}$ of $G_i$ and $\mathcal{L}$ of $L$ such that the line bundle
			\[\mathcal{A}_{\mathcal{Y}}:=\widetilde{h}^*\mathcal{L}{\big(}-\sum_{i} a_i \mathcal{G}_{i}\big)\]
			satisfies $\mathcal{A}_{\mathcal{Y}}|_{\mathcal{E}} \sim 0$.
		\end{claim}
		\begin{proof}[Proof of the Claim]

			We divide the proof according to the classification of \autoref{l-logcanonical-CY}.\\
			
			\emph{Case (i)}. 
			Suppose $Z \cong \mathbb{P}^2_k$ and $E_Z$ is a globally $F$-split elliptic curve. In particular, $E$ is also an elliptic curve.
			As $h$ is crepant, $Y$ is obtained by blowing-up points $p_1, \dots, p_r$ on $E_Z$ respectively $n_1, \dots, n_r$ times\footnote{ specifically, we first blow-up $p_1 \in E_Z$, then (if $n_1>1$) we blow-up the intersection of $E_Z$ with the exceptional divisor of this blow-up, and repeat this procedure until there are exactly $n_1$ exceptional curves over $p_1$, after which we do the same $n_2$-times for $p_2 \in E_Z$, and so on so forth.}. 
   
		 We start by constructing a lifting of $(Y, h^{-1}_*E_Z + \Ex(h))$:
		\begin{enumerate}
		    \item[(a)] first, we take the canonical lifting $\mathcal{E}_{\mathcal{Z}}$ given by \autoref{t-ms};
		    \item[(b)] second, we set ${\mathcal Z} := \mathbb{P}^2_{W(k)}$ and consider the embedding $\mathcal{E}_{\mathcal{Z}} \subset \mathcal{Z}$ given by $|\mathcal{O}_{\mathcal{E}_{\mathcal{Z}}}(3\widetilde{O})|$, where $\widetilde{O}$ is the origin of the elliptic scheme;
		    \item[(c)] last, we take the canonical  liftings $\widetilde{p}_i \in \mathcal{E}_{\mathcal{Z}}$ given by \autoref{lem: lift_points}, and construct $\mathcal{Y}$ by blowing-up the points $\widetilde{p}_i$ on $\mathcal{Z}$ exactly $n_i$-times.
		\end{enumerate}
			Here $({\mathcal Z}, \mathcal{E}_{\mathcal{Z}})$ is a lifting of $(Z,E_Z)$. Let $\widetilde{h} \colon \mathcal Y \to \mathcal Z$ denotes the composition of blow-ups. Since we have blown-up smooth points only, we get that $(\mathcal{Y}, \widetilde{h}^{-1}_*\mathcal{E}_Z + \Ex(\widetilde{h}))$ is a lifting of $(Y, h^{-1}_*E_Z + \Ex(h))$.
			
			Now we prove that $\mathcal{A}_{\mathcal{Y}}|_{\mathcal E} \cong \mathcal{O}_{\mathcal{E}}$. First, it is easy to see that $A_Y|_E \cong \mathcal{O}_E(3d O-\sum m_i p_i)$ for some $m_i>0$ and $d>0$.
			As $\mathcal{Z}=\mathbb{P}^2_{W(k)}$ there exists a unique lifting $\mathcal{L}$ of $L$.
			By the choice of the liftings $\widetilde{p}_i$, we obtain that $\mathcal{A}_{\mathcal{Y}}|_{\mathcal{E}} \cong \mathcal{O}_{\mathcal{E}}(3d \widetilde{O} - \sum m_i \widetilde{p_i})$ is the canonical lifting of the trivial line bundle, thus trivial itself.
		
	    We are left to check that for every irreducible divisor $F_i \subset F$, there is a lifting $\mathcal{F}_i \subset \mathcal{Y}$.
		First we claim that $H^1(Y, \mathcal{O}_Y(F_i-E))=0$.  Since $-E \sim K_Y$, by Serre duality it is sufficient to show the vanishing of $H^1(Y,\mathcal{O}_Y(-F_i))$. But this is clear from the exact sequence 
		\[
		H^0(Y, \mathcal{O}_Y) \to H^0(F_i, \mathcal{O}_{F_i}) \to H^1(Y, \mathcal{O}_Y(-F_i)) \to H^1(Y,\mathcal{O}_Y)=0.
		\]
        Consider now the line bundle $L_i := \mathcal{O}_Y(F_i)$ and write $F_i \sim \pi^*\pi_*F_i + \sum b_j G_j = e\pi^*H + \sum b_jG_j$, where $H \subseteq \mathbb{P}^2_k$ is a line and $e \in \mathbb{Z}_{\geq 0}$. 
        We define a lifting of $L_i$ by $\mathcal{L}_{{i}} := \pi^*\mathcal{O}_{\mathbb{P}^2}(e)  \otimes \mathcal{O}_{Y}(\sum b_j G_{ j})$.
        By construction $\mathcal{L}_{{i}}|_{E}$ is the canonical lift of $L_{i}|_\mathcal{E}$, and so $\mathcal{L}_{{i}}|_{\mathcal{E}}$ is trivial.
        Now consider the exact sequence:
        \[ H^0(\mathcal{Y}, \mathcal{L}_{i}) \to H^0(E, \mathcal{L}_{{i}}|_{E}) \to H^1(\mathcal{Y}, \mathcal{L}_{i}(-E)).\]
        The middle term is a free $W(k)$-module of rank one as $\mathcal{L}_{i}|_{E}$ is trivial.
        The right term is zero by semicontinuity as $\mathcal{L}_{i}(-E)$ is a  lift of $\mathcal{O}_Y(F_i-E)$ whose first cohomology group vanishes as shown above. 
        Therefore $H^0(\mathcal{Y}, \mathcal{L}_{i}) \neq 0$ and its non-zero section yields a lift of $F_i$.\\
        
		\emph{Case (ii)}.
			Suppose $Z\cong \mathbb{P}^2_{k}$ and $E_Z$ is a union of three lines in general position, so up to an automorphism $E_Z=(xyz=0)$.
			There is a factorisation of crepant birational morphism 
			\[(Y,E) \xrightarrow{\varphi} (W,E_W) \xrightarrow{\psi} (\mathbb{P}^2_k, E),\]
			where $\psi$ is the composition of blow-ups at closed points belonging to two irreducible components of $E$, while the centres of $\varphi$ are those lying in only one irreducible component.
			Note $W$ is a projective toric variety and $E_W$ is the toric boundary divisor. 
			We consider the unique toric lifting $\widetilde{\psi}\colon (\mathcal{W}, \mathcal{E}_{\mathcal{W}})\to (\mathbb{P}^2_{W(k)},\mathcal{E}_{\mathcal{Z}})$ over $W(k)$.  Again, as $\mathcal{Z}=\mathbb{P}^2_{W(k)}$, there exists a unique lifting $\mathcal{L}$ of $L$.
			We thus reduced to the case where $(Z, E_Z=\sum_i E_{Z,i})$ is a smooth toric surface pair and $h$ is the blow-up of the points $\left\{p_{ij} \in E_{Z,i}\setminus(\cup_{l\neq i}E_{Z,l})\right\}_{i,j}$ repeated $n_{ij}$-times, where we follow the notation of \autoref{l-Pic-cycle rational}.
			Let $(\mathcal{Z}, \mathcal{E}_{\mathcal{Z}})$ be the toric lifting over $W(k)$.
			For any $p_{ij} \in E_{Z,i}$ we consider the Teichm\"{u}ller lifting $\omega(p_{ij}) \in \mathcal{E}_{\mathcal{Z},i}$ and we construct $\mathcal{Y}$ as the blow-up along $\omega(p_{ij})$ repeated $n_{ij}$ times.
			As $A_Y|_E=\sum m_{ij} p_{ij}$ for some $m_{ij}$ and $\lambda(A_Y|_E)=1$,
			we deduce that $A_{\mathcal{Y}}|_{\mathcal{E}}= \sum m_{ij} \omega(p_{ij})$. 
			By \autoref{l-Pic-cycle rational} and multiplicativity of the Teichm\"uller morphism, we conclude $\lambda(\mathcal{O}_{\mathcal{E}_K}(\sum m_{ij} \omega(p_{ij})))=\omega(\mathcal{\lambda}(\mathcal{O}_E(\sum m_{ij}p_{ij})))=1$ and thus we conclude  $\mathcal{A}_{\mathcal{Y}}|_{\mathcal{E}_K} \sim 0$.
            By Grauert's theorem, then $\mathcal{A}_{\mathcal{Y}}|_{\mathcal{E}}$ is trivial.
			
			 We can repeat the same proof as in Case (i) (replacing $\mathbb{P}^2_k$ with the toric variety $W$) to show that for every irreducible divisor $F_i \subset Y$, there is a lifting $\mathcal{F}_{i} \subset \mathcal{Y}$.\\
			 
		\emph{Case (iii)}. Suppose $Z \cong \mathbb{P}_B(M \oplus N)$, together with the projection $p \colon Z \to B$ and $E_Z=C+D$ is the union of two disjoint sections. We denote by $\mathcal{O}_Z(1)$ the natural Serre line bundle.
		As $E_Z$ is globally $F$-split, $B$ is also a globally $F$-split elliptic curve. 
		We consider the canonical lifting $\mathcal{B}$ over $W(k)$ together with the canonical lifting $\mathcal{M}$ (resp.\ $\mathcal{N}$) of $M$ (resp.\ $N$) given by \autoref{t-ms}. The functoriality of the canonical liftings shows that the sections $\mathcal{C}$ and $\mathcal{D}$ induced by $\mathcal{M}$ (resp. $\mathcal{N}$) are the canonical liftings of $C$ (resp. $D$).
		We choose the lifting $ \widetilde{p} \colon (\mathcal{Z}, \mathcal{E}_{\mathcal{Z}}):=(\mathbb{P}_{\mathcal{B}}(\mathcal{M} \oplus \mathcal{N}), \mathcal{C}+\mathcal{D}) \to \mathcal{B}$. 
		We can lift $L$ in a canonical way to $\mathcal{Z}$ as follows. As $\Pic(Z)=\pi^*\Pic(B) \oplus \mathbb{Z}[\mathcal{O}_Z(1)]$, there exists $n \in \mathbb{Z}$ such that $L \cong p^*H \otimes \mathcal{O}_{Z}(n)$, where $H \in \Pic(B)$.
		We consider the lifting $\mathcal{L}=p^*\mathcal{H} \otimes \mathcal{O}_{\mathcal{Z}}(n)$.
		We can now repeat the same proof as in the case $(i)$ by blowing-up the canonical lifts of the points to end the proof.
			
		Note that every irreducible component $F_i$ of $F$ is contained in a fibre. As $Y \to Z$ is a composition of blow-ups, it is easy to see that $F_i$ lifts to $\mathcal{F}_i \subset \mathcal{Y}$.
	\end{proof}
	
	Let $\varphi \colon (Y,E+F) \to (T,E_T)$ be the contraction of the trees of $(-2)$-curves given by $F$. 
	Note there is a birational contraction $\psi \colon (T,E_T) \to X$, contracting exactly $E_T$.
    Let $(\mathcal{Y}, \mathcal{E} + \mathcal{F})$ be the lifting constructed in \autoref{claim-good-lifting}.
    By \autoref{p-lc-nottriv} we can contract $\mathcal{F}$ to get a lifting $\widetilde{\varphi} \colon (\mathcal{Y}, \mathcal{E} + \mathcal{F}) \to (\mathcal{T}, \mathcal{E}_{T})$ of $\varphi$. 
	Since $A_Y|_F \sim 0$ and $H^1(F, \mathcal{O}_F)=0$ as $F$ is a tree of smooth rational curves, we deduce that $\mathcal{A}_{\mathcal{Y}}|_{\mathcal{F}} \sim 0$, and so $\mathcal{A}_{\mathcal{Y}}$ is $\widetilde{\varphi}$-trivial.
	Therefore it descends to a line bundle $\mathcal{A}_{\mathcal{T}}$ on $\mathcal{T}$. 
    As $\mathcal{A}_{\mathcal{T}}|_{\mathcal{E}_{T}} \sim 0$, by \autoref{p-lifting-criterion} we conclude there exists a lifting $\widetilde{\psi}$ of $\psi$. 
    Thus $\widetilde{f} = \widetilde{\psi} \circ \widetilde{\varphi}$ is the desired lifting of $f$.
    \end{proof}
	
\begin{remark}
    The toric lifting of the toric pair used to solve case (ii) of \autoref{claim-good-lifting} can be thought as a canonical lifting as it is the unique lifting admitting a lifting of the Frobenius morphism compatible with the toric boundary (as defined in \cite{AWZII}).
\end{remark}
	
We can finally prove the main result of this article.

\begin{theorem}\label{t-3}
	Let $(X,D)$ be a normal projective globally $F$-split surface pair, where $D$ is a reduced Weil divisor. 
	Then $(X,D)$ is strongly liftable over $W(k)$.
\end{theorem}	
	
\begin{proof}
	If $D=0$ and $K_X \sim 0$, this is proven in \autoref{p-algebraicity-k-trivial} and \autoref{p-strictlyLC-case}.
	The remaining cases are proven in \autoref{p-lc-nottriv}.
\end{proof}
	
\section{Applications}\label{s-applications}
	
	In this section we show some applications of our results to the study of singularities on globally $F$-split surfaces and to the existence of special liftings of del Pezzo type globally $F$-split surfaces.

	\subsection{Singularities of globally $F$-split surfaces}
	
	The following result allows to compare the singularities of a variety admitting a log lifting with those in characteristic zero.  
    We recall the definition of a weighted dual graph of the exceptional locus for a surface singularity.

\begin{definition}\label{def-weighted-dual-graph} 
    Let $X$ be a normal projective surface over $k$ and let $f\colon (Y,\Ex(f)) \to X$ be a log resolution.
    The \emph{weighted dual graph} of $E$ is the graph whose vertices $\left\{v_i\right\}$ correspond to irreducible components $E_i$ of $\Ex(f)$ and two vertices $v_i$ and $v_j$ are connected by an edge for each of the intersection point of $E_i$ and $E_j$.
    Moreover, every vertex $v_i$ is labelled with the self-intersection $E_i^2$ and the genus $g(E_i)$.
\end{definition}
    	
	\begin{proposition}\label{p-good-lifting}
		Let $X$ be a normal projective surface over $k$ and let $f\colon (Y,E) \to X$ be a log resolution.
		Suppose there exists a projective lifting $\tilde{f}\colon (\mathcal{Y}, \mathcal{E}) \to \mathcal{X}$ of $f \colon (Y,E) \to X$ over $W(k)$.
		Then
		\begin{enumerate}
			\item[(a)] the weighted dual graph of $E$ is equal to that of $\mathcal{E}_{\overline{K}}$;
			\item[(b)] if $X$ has rational singularities, then  $\mathcal{X}$ has rational singularities;
			\item[(c)] if $X$ has klt singularities, then $\mathcal{X}$ is klt;
			\item[(d)] if $\Pic(\mathcal{Y}) \to \Pic(Y)$ is surjective and $X$ has rational singularities, then $\rho(X)=\rho(\cX_{\overline{K}})$.
		\end{enumerate}
	\end{proposition}

	\begin{proof}
	We define $E\coloneqq \sum_{i} E_i$ and $\mathcal{E}=\sum_{i} \mathcal{E}_i$, where each $E_i$ is an irreducible component of $E$ and each $\mathcal{E}_{i}$ is a lifting of $E_i$.
	Let us show that $\Ex(\tilde{f})=\mathcal{E}$.
	To this end, it suffices to prove that $\Ex(\tilde{f}_{K})=\cE_{K}$.
	Let $\mathcal{A}$ be an ample divisor on $\mathcal{X}$.
	We denote the pull-back $\mathcal{L}:=\tilde{f}^*\mathcal{A}$ and its restriction to the closed fibre by $L:=\cL\otimes_{W(k)}k$.
	Take an irreducible component $\cE_{i,K}$ of $\cE_{K}$. Then $\cE_{i,K}\cdot \cL_{K}=E_{i}\cdot L=0$ and thus $\cE_{i,K}\subset \Ex(\tilde{f}_{K})$.
	Next, let $\cF_K$ be a prime divisor contained in $\Ex(\tilde{f}_{K})$, let $\cF$ be its closure in $\cY$, and let $F:=\cF\otimes_{W(k)}k$. 
	Then $L\cdot F=\cL_{K}\cdot \cF_K=0$  and thus $F=\sum_{i} m_{i}E_{i}$ for some $m_{i} \geq 0$.
	By the negativity lemma, there exists $E_j$ such that $F \cdot E_j<0$.
	Now $\cF_K\cdot \cE_{j,K} = F \cdot E_{j} <0$ and hence $\cF_K=\cE_{j,K}$.
	Thus we deduce $\Ex(\tilde{f}_{K})=\cE_{K}$ and $\Ex(\tilde{f})=\mathcal{E}$.
		
	In this paragraph, we show the assertion (a).
	We start by proving that $\mathcal{X}$ is normal.
	Since $X$ is $S_2$ and it is a Cartier divisor of $\mathcal{X}$, it follows that $\mathcal{X}$ is $S_3$ by \cite[Proposition 5.3]{km-book}.
	Furthermore, $\mathcal{X}$ is regular outside $\tilde{f}(\mathcal{E})$, which is a closed subset of codimension at least 2 hence $\mathcal{X}$ is $R_1$ and thus normal.
	Since $E$ and $\cE_{\overline{K}}$ have the same intersection matrix, we obtain assertion (a).
		
		As for (b), we 
		consider the short exact sequence $0 \to \mathcal{O}_\mathcal{Y}(-Y) \to \mathcal{O}_{\mathcal{Y}} \to \mathcal{O}_Y \to 0 $.
		As $\tilde{f} \colon \mathcal{Y} \to \mathcal{X}$ is a simultaneous resolution of the family $\mathcal{X}$, by hypothesis $R^i f_*\mathcal{O}_Y=0$ for $i>0$ and therefore we deduce that the multiplication $R^i \tilde{f}_* \mathcal{O}_\mathcal{Y}(-Y) \to R^i \tilde{f}_*\mathcal{O}_{\mathcal{Y}} $ is an isomorphism, concluding $R^i \tilde{f}_* \mathcal{O}_\mathcal{Y}=0$ by Nakayama's lemma.
        
        To prove (c), we know $K_Y+\sum a_i E_i =f^*K_X$ for $a_i<1$ by the klt hypothesis.
        Note that by (a), the surface $\mathcal{X}_K$ has klt singularities (as being klt can be checked from the dual graph for surfaces).
        As $K_{\mathcal{Y}_K}+\sum a_i\mathcal{E}_{i,K}$ is $f_K$-numerically trivial and $X_K$ has klt singularities, by the base-point-free theorem we conclude $K_{\mathcal{Y}_K}+\sum a_i\mathcal{E}_{i,K} \sim_{\mathbb{Q}, X} 0$. 
        By \cite[Theorem 1.2]{Wit21}, we conclude that $K_{\mathcal{Y}}+\sum a_i \mathcal{E}_i \sim_{\mathbb{Q}, \mathcal{X}} 0,$ so $K_{\mathcal{X}}=\tilde{f}_* (K_{\mathcal{Y}}+\sum a_i \mathcal{E}_i)$ is $\mathbb{Q}$-Cartier.
        As $\mathcal{Y} \to \mathcal{X}$ is a log resolution and $a_i<1$, we conclude $\mathcal{X}$ is klt.

		We now prove (d). By \cite[Proposition 3.6]{MP12}, we have that $\rho(Y) \geq \rho(\cY_{\overline{K}})$. 
		Since $\Pic(\mathcal{Y}) \to \Pic(Y)$ is surjective and $\Pic(\mathcal{Y})\cong \Pic(\mathcal{Y}_K)$ by \cite[Expos\'e X, Appendix 7.8]{SGA6}, we conclude $\rho(Y) \leq \rho(\cY_{\overline{K}})$, and so $\rho(Y)=\rho(\cY_{\overline{K}})$.
	    Let $n$ be the number of exceptional divisors in $\Ex(f)$. By (a), $n$ is also the number of exceptional divisors in $\Ex(\widetilde{f}_{\overline K})$. Thus, as $X$ and $\mathcal{X}_K$ are $\mathbb{Q}$-factorial by \cite[Proposition 10.9]{kk-singbook}, 
		\[
		\rho(X) = \rho(Y) - n = \rho(\mathcal{Y}_{\overline K}) - n = \rho(\mathcal{X}_{\overline K}),
		\]
		which concludes the proof of (d).
\end{proof}
	
	As an application we show the existence of a lifting of a globally $F$-split surface $X$ over $W(k)$ which preserves the Picard rank and the type of the singularities of $X$.
	
	\begin{proof}[Proof of \autoref{c-1}]
		We pick $f \colon (Y,E) \to X$ to be a log resolution, and take a lifting $(\mathcal{Y}, \mathcal{E}) \to \mathcal{X}$ of $f \colon (Y,E) \to X$ over $W(k)$ granted by \autoref{t-3}. 
		If $H^2(X,\mathcal{O}_X)=0$, then $\Pic(\mathcal{Y}) \to \Pic(Y)$ is surjective by \cite[Corollary 8.5.5]{FAG}, so we conclude by \autoref{p-good-lifting}.
		Suppose that $H^2(X, \mathcal{O}_X) \neq 0$. Then $H^0(X, \mathcal{O}_X(K_X)) \neq 0$ and thus $X$ is a globally $F$-split surface with $K_X \sim 0$.  
		We consider the canonical lifting $ \widetilde{f} \colon (\mathcal{Y}_{\can}, \mathcal{E}_{\can}) \to \mathcal{X}_{\can}$ constructed in \autoref{p-algebraicity-k-trivial} and \autoref{p-strictlyLC-case}. Again we apply \autoref{p-good-lifting}.
	\end{proof}
	As a consequence of \autoref{c-1} we deduce an explicit bound on the Gorenstein and global index of globally $F$-split klt $K$-trivial surfaces. We recall their definitions. 

	\begin{definition} \label{def:definition-of-global-index}
	    Let $X$ be a normal $\Q$-Gorenstein variety. The \textit{Gorenstein index of $X$} is the smallest integer $m>0$ such that $mK_X$ is Cartier.
	    If $X$ is projective and $K_X\sim_\Q 0$, the \textit{global index of $X$} is the smallest integer $n>0$ such that $nK_X \sim 0$.
	\end{definition}
	
	\begin{proof}[Proof of \autoref{c-index-CY}]
		By \autoref{c-1}, there exists a lifting $\mathcal{X}$ over $W(k)$ such that $\mathcal{X}_{\overline{K}}$ is a klt projective surface over an algebraically closed field of characteristic zero whose weighted dual graph of the minimal resolution of $\mathcal{X}_{\overline{K}}$ is the same as that of $X$. We claim that $K_{\mathcal{X}} \equiv 0$: take an irreducible curve $C_K \subset \mathcal{X}_K$.
        Let $\mathcal{C}$ be the closure of $C_K$ in $\mathcal{X}$ and let $C=\mathcal{C} \otimes_{W(k)} k$. Then
        $K_{\mathcal{X}_K} \cdot C_K = K_X \cdot C=0$. Thus $K_{\mathcal{X}} \equiv 0$.
        
		We note that the Gorenstein index of a klt surface $S$ is determined by the weighted dual graph of the minimal resolution $\pi \colon T\to S$ as follows.
		We can write $K_T+\sum_{i}a_iE_i=\pi^{*}K_S$ for some $a_i\in\Q_{>0}$.
		Since $nK_S$ is Cartier if and only if $\pi^{*}nK_S$ is Cartier by \cite[Lemma 2.1]{CTW17}, the Gorenstein index of $S$ is equal to $\min\{n\in\Z_{>0} \mid na_i\in\Z\,\,\textup{for}\,\,\textup{all}\,\,i\}$.
		Since $\mathcal{X}_{\overline{K}}$ has Gorenstein index at most 21 by \cite[Theorem C]{Bla95}, so does $X$.
		
		Finally, we show the assertion about the global index of $K_X$.
		If $X$ has non-canonical singularities, then the global index of $K_X$ coincides with the Gorenstein index by~\cite[Lemma 3.12]{Kaw21} and in particular it is at most $21$.
		On the other hand, if $X$ has only canonical singularities, then the global index is at most 6 by \cite[Theorem 1]{BM77}.
		Thus the assertion holds.
	\end{proof}

\subsection{Lifting globally $F$-split del Pezzo and Calabi--Yau pairs}
	
	In what follows, we show that we can always choose a lifting  of a globally $F$-split surface of del Pezzo type over $W(k)$ so that it is still a surface of del Pezzo type.
	
	\begin{lemma}\label{l-Zariski-decomp-dP}
		Let $X$ be a surface of del Pezzo type and let $f \colon Y \to X$ be the minimal resolution.
		Then there exists an effective $\Q$-divisor $D$ on $Y$ such that $\Supp(D)$ is snc, $\Ex(f) \subseteq \Supp(D)$, and the pair $(Y,D)$ is log del Pezzo.
	\end{lemma}
	
	\begin{proof}
		Since $f$ extracts only divisors with non-positive discrepancies, the anti-canonical rings of $Y$ and $X$ coincide.
		By \cite[Lemma 2.2]{ABL20}, $Y$ is a Mori dream space and there is a factorisation \[\pi \colon Y \to X \to Z:=\Proj \bigoplus_{m \geq 0} H^0(X, \mathcal{O}_X(-mK_X)).\]
		By \cite[Lemma 2.9]{BT22}, $Z$ is a klt del Pezzo surface, thus we have that $\Supp(\Ex(\pi))$ is snc by the classification results of \cite[Section 3.40]{kk-singbook}.
		
		We write $K_Y \sim_{\Q} \pi^*K_Z - F$ where $F$ is effective and it is contained in the $\pi$-exceptional locus.
		Thus we have that $(Y, \Supp(F))$ is snc, $(Y,F)$ is klt and $-(K_Y+F)$ is a big and nef $\Q$-Cartier divisor and that its null locus is contained in $\Ex(\pi)$.
		
		Let $A$ be an ample effective divisor on $Y$ and define $H:=\pi^*\pi_*A-A$.
		Note that $-H$ is $\pi$-ample and that $\Supp(H)$ coincides with $\Ex(\pi)$.
		Finally for sufficiently small $\varepsilon>0$, $(Y, F+\varepsilon H)$ is klt and $-(K_Y+F+\varepsilon H)$ is ample.
	\end{proof}
	
	\begin{theorem}\label{t-lift-GFSFanosurf}
		Let $X$ be a globally $F$-split surface of del Pezzo type.
		Let $f\colon (Y,E) \to X$ be its minimal resolution pair. 
		Then there exists a lifting $\tilde{f}\colon (\mathcal{Y}, \mathcal{E}) \to \mathcal{X}$ of $f$ over $W(k)$ such that
		\begin{enumerate}
			\item[(a)]   $\mathcal{X}$ is a normal threefold with klt and rational singularities and $\Ex(\tilde{f})=\mathcal{E}$;
			\item[(b)]   $\rho(X)=\rho(\cX_{\overline{K}})$ and the dual graph of $\Ex(f)$ is equal to $\Ex(\tilde{f}_{\overline{K}})$;
			\item[(c)] $\cY_K$  and $\cX_K$ are surfaces of del Pezzo type.
		\end{enumerate}
	\end{theorem}
	\begin{proof}
		Since $X$ is a surface of del Pezzo type we can apply  \autoref{l-Zariski-decomp-dP} to find an effective $\Q$-divisor $D$ on $Y$ such that $\Supp(D)$ is snc, it contains $E$ and $(Y,D)$ is log del Pezzo.
		Then  by \autoref{p-liftpair-W2} and \autoref{p-lift-Fanopairs} there exists a lifting $(\mathcal{Y}, \Supp(\mathcal{D}))$ over $W(k)$.
		As klt surface singularities are rational, and $H^i(Y,\cO_Y)=0$ for $i>0$ by \cite[Lemma 5.1]{Ber21}, the morphism $f$ lifts to $\widetilde{f}$ by \autoref{p-extending-algebraisable}.
		Since ampleness is an open condition in families, the pair $(\cY_K, \cD_K)$ is a log del Pezzo pair.
		Assertions (a) and (b) then follow from \autoref{p-good-lifting}, while (c) is a consequence of \cite[Lemma 2.9]{BT22}.
	\end{proof}

We can now prove the Bogomolov bound on the singular points of klt del Pezzo surfaces.
	
	\begin{proof}[Proof of \autoref{c-bogomolovbound}]
		Let $f \colon (Y,E) \to X$ be the minimal resolution and consider the lifting $\widetilde{f} \colon (\mathcal{Y}, \mathcal{E}) \to \mathcal{X}$ over $W(k)$ given by \autoref{t-lift-GFSFanosurf}. As $-K_{\mathcal{X}_K}$ is an ample $\mathbb{Q}$-Cartier divisor, we conclude by the characteristic zero bound proven in \cite[Theorem 1.2]{LX21} and \autoref{c-1}.
	\end{proof}
	
	\begin{remark}By \cite[Theorem 5.1]{SS10} a variety of Fano type over characteristic zero has globally $F$-regular (in particular $F$-split) type.
	We just proved an inverse direction in dimension two: given a globally $F$-split surface of del Pezzo type, we can construct a lifting to characteristic zero which remains of del Pezzo type.
	The following example shows however that a general lift is not a surface of del Pezzo type.
    \end{remark}
	
	\begin{example}
		Let $e>0$ be an integer number such that $q=p^e>10$, and we fix $k=\mathbb{F}_q$. Consider the smooth $W(k)$-scheme $\mathcal{X}:=\mathbb{P}^2_{W(k)}$ and choose $\mathcal{P}_1, \dots, \mathcal{P}_9$ distinct smooth $W(k)$-sections such that 
		\begin{enumerate}
			\item[(a)] $\mathcal{P}_{1,K}, \dots, \mathcal{P}_{9,K}$ are in general position;
			\item[(b)] $\mathcal{P}_{1,k}, \dots, \mathcal{P}_{9,k}$ are distinct points lying on a $k$-line $L$.
		\end{enumerate}
		Let $\widetilde{\pi} \colon \mathcal{Y} \to \mathcal{X}$ be the blow-up along $\mathcal{P}_1, \dots, \mathcal{P}_9$.  
		We now check that $Y:= \mathcal{Y} \otimes_{W(k)}  k$ is globally $F$-split. 
		Let $\pi:=\widetilde{\pi}\vert_Y$: as $(\mathbb{P}^2_k, L)$ is globally $F$-split and $\pi^*(K_{\mathbb{P}^2_k} +L)=K_{Y}+\pi_*^{-1}L$, the pair $(Y, \pi_*^{-1}L)$ is globally $F$-split by \autoref{l-GFR-pullback} and it is a surface of del Pezzo type.
		However, $\mathcal{Y}_K$ is not a surface of del Pezzo type: indeed, as $\left\{\mathcal{P}_{i,K}\right\}_{i=1}^9$ are in general position, the divisor $-K_{\cY_K}$ is not even big.
	\end{example}
	
	We conclude by showing the existence of a lifting of globally $F$-split log Calabi--Yau surface pairs with log Calabi--Yau total space. The main difficulty is to prove the log canonical divisor of the total space is $\mathbb{Q}$-Cartier, for which we use the existence of a log lifting.
	
	\begin{theorem}
		Let $(X,D)$ be a globally $F$-split surface pair such that $D$ is reduced and $K_X+D \sim_{\mathbb{Q}} 0$. Then there exists a log canonical pair $(\mathcal{X}, \mathcal{D})$ lifting $(X,D)$ over $W(k)$ such that $K_{\mathcal{X}}+\mathcal{D}\sim_{\mathbb{Q}} 0$.
	\end{theorem}
	
	\begin{proof}
	Let $f\colon (Y,D_Y+E) \to (X, D)$ be a dlt model which admits a lifting \[\widetilde{f} \colon (\mathcal{Y}, \mathcal{D}_{\mathcal{Y}}+ \mathcal{E}) \to (\mathcal{X},\mathcal{D})\] 
 over $W(k)$ given by \autoref{t-3}. As $K_{\mathcal{Y}}+\mathcal{D}_{\mathcal{Y}}+\mathcal{E} \equiv 0$ and it is a dlt pair, 
 we have \[K_{\mathcal{Y}_K}+\mathcal{D}_{\mathcal{Y}_K}+\mathcal{E}_K \sim_{\mathbb{Q}} 0\] by the abundance theorem for log canonical surfaces \cite[Corollary 1.2]{Fuj12}.  As $K_{\mathcal{Y}}+\mathcal{D}_{\mathcal{Y}}+\mathcal{E}$ is nef over $W(k)$, this implies $K_{\mathcal{Y}}+\mathcal{D}_{\mathcal{Y}}+\mathcal{E}\sim_{\Q}0$, hence $K_{\mathcal{X}}+\mathcal{D} =\widetilde{f}_*(K_{\mathcal{Y}}+\mathcal{D}_{\mathcal{Y}}+\mathcal{E})  \sim_\Q  0$. As $\widetilde{f} \colon (\mathcal{Y}, \mathcal{D}_{\mathcal{Y}}+ \mathcal{E}) \to (\mathcal{X},\mathcal{D})$ is crepant, we conclude the pair $(\mathcal{X},\mathcal{D})$ has log canonical singularities.
	\end{proof}

  	\bibliographystyle{amsalpha}
	\bibliography{refs}
\end{document}